\pdfoutput=1
\documentclass{amsart}
\usepackage{graphicx} 
\usepackage{amsmath,amsthm,amssymb} 
\usepackage{mathtools}
\usepackage{mathrsfs}
\usepackage{stmaryrd}
\usepackage{enumerate} 
\usepackage{xcolor}
\usepackage{comment}
\usepackage{tikz-cd}
\usepackage[left=1.2in, right=1.2in, top=1.3in, bottom=1.3in]{geometry}
\usepackage{hyperref}
\hypersetup{colorlinks=true, citecolor=blue, urlcolor=blue}
\usepackage[
style=alphabetic, backref=true, backrefstyle=none
]{biblatex}
\newtheorem{theorem}{Theorem}[section]
\newtheorem{lemma}[theorem]{Lemma}
\newtheorem{prop}[theorem]{Proposition}
\newtheorem{cor}[theorem]{Corollary}

\theoremstyle{definition}

\newtheorem{defn}[theorem]{Definition}

\theoremstyle{remark}

\newtheorem{rem}[theorem]{Remark}

\newcommand{\A}[2]{\mathbb{A}_{#1}^{#2}}
\newcommand{\Q}{\mathbb{Q}}
\newcommand{\Z}{\mathbb{Z}}

\newcommand{\C}{\mathbb{C}}
\newcommand{\R}{\mathbb{R}}

\newcommand{\GL}{\operatorname{GL}}
\newcommand{\Ga}{\mathbb{G}_a}
\newcommand{\Gm}{\mathbb{G}_m}

\newcommand{\pp}{\mathfrak{p}}

\DeclareMathOperator{\Hom}{Hom}

\DeclareMathOperator{\Spec}{Spec}

\DeclareMathOperator{\Sp}{Sp}
\DeclareMathOperator{\Spf}{Spf}

\DeclareMathOperator{\Res}{Res}
\DeclareMathOperator{\Sym}{Sym}
\DeclareMathOperator{\Fil}{Fil}
\DeclareMathOperator{\Gr}{Gr}
\DeclareMathOperator{\Ha}{Ha}
\DeclareMathOperator{\Hdg}{Hdg}
\DeclareMathOperator{\can}{can}

\DeclareMathOperator{\id}{id}
\DeclareMathOperator{\dR}{dR}
\DeclareMathOperator{\Tor}{Tor}
\DeclareMathOperator{\Tr}{Tr}

\renewcommand{\le}{\leqslant}
\renewcommand{\leq}{\leqslant}
\renewcommand{\ge}{\geqslant}
\renewcommand{\geq}{\geqslant}

\title{{$p$}-adic Asai and twisted triple product {$L$}-functions for finite slope families}
\author{Ananyo Kazi}

\author{David Loeffler}

\address{Faculty of Mathematics and Computer Science, UniDistance Suisse, Schinerstrasse 18, CH-3900 Brig, Switzerland}
\email{ananyo.kazi@unidistance.ch ; david.loeffler@unidistance.ch}
\urladdr{\url{https://orcid.org/0000-0002-7609-968X} ; \url{https://orcid.org/0000-0001-9069-1877}}

\thanks{Supported by the European Research Council through the Horizon 2020 Excellent Science programme (Consolidator Grant ``Shimura varieties and the BSD conjecture'', grant ID 101001051).}
\date{Compiled \today}

\begin{document}

\begin{abstract}
 We define a two-variable $p$-adic Asai $L$-function for a finite-slope family of Hilbert modular forms over a real quadratic field (with one component of the weight, and the cyclotomic twist variable, varying independently); and a two-variable ``twisted triple product'' $L$-function, interpolating the central $L$-value of the tensor product of such a family with a family of elliptic modular forms. The former construction generalizes a construction due to Grossi, Zerbes and the second author for ordinary families; the latter is a counterpart of the twisted triple product $L$-function of \cite{kazi2024twistedtripleproductpadic}, but differs in that it interpolates classical $L$-values in a different range of weights, in which the dominant weight comes from the Hilbert modular form. Our construction relies on a ``nearly-overconvergent'' version of higher Coleman theory for Hilbert modular surfaces.
\end{abstract}

\maketitle

\tableofcontents

\newcommand{\Sh}[1]{\mathcal{#1}}
\newcommand{\V}[2]{\mathbb{V}_0(#1,#2)}
\newcommand{\Hs}{\Sh{H}^{\sharp}_{\Sh{A},1}}
\newcommand{\Os}{\omega^{\sharp}_{\Sh{A},1}}

\section{Introduction}

 In this paper, we study the $p$-adic interpolation of critical values of the \emph{Asai}, or \emph{twisted tensor product}, $L$-function associated to Hilbert modular eigenforms over a real quadratic field $F$; and the \emph{twisted triple product} associated to the product of a quadratic Hilbert modular form and an elliptic modular form.

 \subsubsection*{Asai case} We build upon an earlier work of Grossi, Zerbes, and the second author \cite{grossi2024padicasailfunctionsquadratic}, which used higher Hida theory to define $p$-adic Asai $L$-functions for Hilbert modular forms of weight $k_1 > k_2$, assuming $p = \pp_1 \pp_2$ is split in $F$ and  our Hilbert eigenform is ordinary (slope 0) at $\pp_1$. Here the use of higher Hida theory -- interpolating not only the sections of automorphic vector bundles, but also their cohomology in positive degrees -- is unavoidable, since the underlying integral formula for the Asai $L$-function involves a Hilbert modular form which is holomorphic at one infinite place but anti-holomorphic at the other, hence contributing to $H^1$ of an automorphic vector bundle but not to $H^0$.
 
 In this paper, we treat the more general ``finite-slope'' case, where the eigenform is not assumed to be ordinary at $\pp_1$ but may have positive slope. This requires the use of higher Coleman theory, in place of higher Hida theory; but an immediate difficulty arises, since higher Coleman theory incorporates an ``overconvergence'' condition (as in traditional Coleman theory), but the family of $\GL_2$ Eisenstein series used in \cite{grossi2024padicasailfunctionsquadratic} is not overconvergent. Rather, it is a \emph{nearly-overconvergent} family in the sense of \cite{andreatta2021triple}. Hence the main technical input in this paper is to develop a ``nearly-version'' of higher Coleman theory in this situation, building on the theory developed in \cite{andreatta2021triple} and \cite{kazi2024twistedtripleproductpadic}. This allows us to define spaces of nearly-overconvergent cohomology for Hilbert modular surfaces, and to construct elements of these spaces as pushforwards of families of nearly-overconvergent modular forms. 
 
 Our main result in this setting is as follows. Let $\Sh{W}$ denote the weight space parametrising characters of $\Z_p^\times$, and let $\underline{\Pi}$ denote a one-parameter family of finite-slope Hecke eigensystems of weight $(k_1 + 2\kappa_{\Sh{U}}, k_2)$ for some fixed $(k_1, k_2)$, and $\kappa_{\Sh{U}}$ the universal character associated to an an affinoid $\Sh{U} \subset \Sh{W}$ around 0. Then we define a meromorphic function $\Sh{L}_{p, \mathrm{As}}^{\mathrm{imp}}(\Pi)(\kappa, \sigma)$ on $\Sh{U} \times \Sh{W}$ (meromorphic in $\kappa$, but holomorphic in $\sigma$ for each fixed $\kappa$), with the following interpolation property: its value at $(\kappa, \sigma) = (a, s) \in \Z^2$, for $a$ such that the specialisation $\Pi[a]$ of $\underline{\Pi}$ at $a$ is classical (which holds for all sufficiently large $a$) and any $s$ with $1 - \frac{k_1 + 2a-k_2}{2} \le s \le \frac{k_1 + 2a-k_2}{2}$, interpolates (up to the usual correction factors) the value at $s$ of the Asai $L$-function of $\Pi[a]$. See Theorem \ref{thm:main} for a precise statement.

 \subsubsection*{Triple-product case} For our second main result, we consider a quadratic Hilbert family $\underline{\Pi}$ as above, and $\underline{\Sigma}$ be a Coleman family of cuspidal eigenforms for $\GL_{2,\Q}$ of weight $\ell + 2b$, with $\ell$ fixed and $b$ varying over an affinoid disc $\Sh{U}' \subset \Sh{W}$ around 0. We assume that for some (or, equivalently, all) classical specialisations $\Pi[a]$ and $\Sigma[b]$ of $\underline{\Pi}$ and $\underline{\Sigma}$, the following conditions hold: the restriction to $\Q$ of the central character of $\Pi[a]$ is the inverse of the central character of $\Sigma[b]$, so the product $\Pi[a] \times \Sigma[b]$ is self-dual; and the local root numbers of $\Pi[a] \times \Sigma[b]$ satisfy the condition of \cite[Theorem 3.2]{michele} at all finite places (which is automatic if the tame level of the $\GL_{2,\Q}$ form is coprime to that of the $\GL_{2,F}$ form and the discriminant). In this setting, we show that there is a rigid-analytic function on $\Sh{U} \times \Sh{U}'$ whose value at $(a, b) \in \Z^2$ such that $k_1 + 2a \ge k_2 + \ell + 2b$ interpolates the square root of the central value of the degree 8 $L$-function $L(\Pi[a] \times \Sigma[b], s)$. Note that this range (in which the first component of the weight of $\underline{\Pi}$ is dominant) is disjoint from the $\GL_{2,\Q}$-dominant range considered in \cite{kazi2024twistedtripleproductpadic}.


 \subsubsection*{Remark} In parallel with the writing of this paper, A.~Graham and R.~Rockwood have independently developed in \cite{grahamrockwood24} a version of ``nearly higher Coleman theory'', in the setting of Siegel threefolds rather than the Hilbert modular surfaces we consider here. Our methods are rather different from theirs, using the formalism of vector bundles with marked sections used in \cite{andreatta2021triple}.
 
\section{General conventions}

 Let $F$ be a real quadratic field, and $p$ a rational prime which is split in $F$. We fix an isomorphism $\bar{\Q}_p \simeq \C$. Let $p = \pp_1\pp_2$, and let $\sigma_i$ be the complex embedding inducing the $\pp_i$-adic valuation under our chosen isomorphism. Let $G = \Res_{F/\Q}\GL_{2,F}$, and let $H$ be the subgroup $\GL_{2,\Q}$, with $\iota: H \to G$ the natural embedding. This factors through the intermediate group $G^* = \{g \in G \,|\, \det(g) \in \Gm \subset \Res_{F/\Q} \Gm\}$.

\section{$\pp_1$-overconvergent cohomology at prime-to-$p$ level}\label{S2}

 \subsection{Setup: Shimura varieties and models}

  \subsubsection{Shimura varieties over $\Q$} Let $K^{(p)}$ be a (sufficiently small) open compact subgroup of $G(\A{f}{(p)})$, and write $K = K^{(p)} \cdot G(\Z_p)$. We write $X_{\Q}$ for a smooth projective toroidal compactification of the Shimura variety for $G$ of level $K$. Note that the Shimura variety $X_{\Q}$ is not geometrically connected, and its geometrically connected components are parametrised by the narrow ray class group $\mathrm{cl}^+_F(K) := F^{\times,+}\backslash \A{F,f}{\times} /\det K$.

  \subsubsection{Moduli spaces} 

   For a fractional ideal $\mathfrak{c}$ in $F$, let $M^{\mathfrak{c}} / \Spec{\Z_{(p)}}$ denote the compactified Hilbert--Blumenthal moduli scheme parametrising tuples $(A, \iota, \lambda, \alpha_{K^{(p)}})$ away from the boundary. Here $A$ is an abelian variety of dimension 2 with a real multiplication structure given by $\iota$, $\lambda$
   is a $\mathfrak{{c}}$-polarisation in the sense of \cite{DelignePappas}, and $\alpha_{K^{(p)}}$ is a $K^{(p)}$-level structure. This scheme is smooth and projective over $\Z_{(p)}$, and its generic fibre is a Shimura variety for $G^*$ when $\mathfrak{c} = \mathfrak{d}^{-1}$ where $\mathfrak{d}$ is the different ideal of $F/\Q$. There is a natural action of $\Sh{O}_F^{\times, +}$ on $M^{\mathfrak{c}}$ given by scaling the polarisation, which factors through the quotient $\Delta_K := \Sh{O}_F^{\times, +}/(K \cap \Sh{O}_F^{\times})^2$; and the scheme
   \[ \bigsqcup_{[\mathfrak{c}] \in \mathrm{cl}^+_F(K)} M^{\mathfrak{c}} / \Delta_K\]
   is a smooth projective $\Z_{(p)}$-scheme whose generic fibre is canonically identified with $X_{\Q}$, so we denote it simply by $X$.

   \begin{rem}
    More canonically, we can avoid choosing representatives of the class group by considering the (non-finite-type) variety $M \coloneqq \bigsqcup_{\mathfrak{c}} M^{\mathfrak{c}}$, where $\mathfrak{c}$ varies over all fractional ideals coprime to $p$. This has a natural action of the group $F^{\times,(p)}_+$ of elements $x \in F^{\times, +}$ which are units at the primes above $p$, given by the isomorphisms
   \[ 
    M^{\mathfrak{c}} \xrightarrow{} M^{x\mathfrak{c}}, \quad (A,\iota,\lambda,\alpha_{K^{(p)}}) \mapsto (A,\iota,x\lambda,\alpha_{K^{(p)}}),
   \]
   and we have $X = M / F^{\times,(p)}_+$.
   \end{rem}
   

\subsection{Automorphic sheaves}

 
Let $\Sh{A} \to {M}$ be the semiabelian scheme extending the universal abelian surface with real multiplication by $\Sh{O}_F$. Let $e \colon {M} \to \Sh{A}$ be the unit section. Let $\underline{\omega}_{\Sh{A}} := e^*\Omega^1_{\Sh{A}/M}$. This is an $\Sh{O}_F \otimes \Sh{O}_M$-line bundle. Let $\Sh{H}_{\Sh{A}}$ be the canonical extension of the relative de Rham sheaf $\Sh{H}^1_{\text{dR}}(\Sh{A}/M)$ over the un-compactified moduli scheme. This is an $\Sh{O}_F \otimes \Sh{O}_M$-module that is locally free of rank 2. There is a short exact sequence known as the Hodge filtration
\[
0 \to \underline{\omega}_{\Sh{A}} \to \Sh{H}_{\Sh{A}} \to \underline{\omega}_{\Sh{A}^{\vee}}^{\vee} \to 0.
\]

Over the connected component $M^{\mathfrak{c}}$ of $M$ that classifies $\mathfrak{c}$-polarised abelian varieties, we have a canonical isomorphism $\underline{\omega}_{\Sh{A}^{\vee}}^{\vee} \simeq \underline{\omega}^{\vee}_{\Sh{A}} \otimes \mathfrak{c}$ using the $\mathfrak{c}$-polarisation \cite[{(1.0.15)}]{KatzNicholasM1978pLfC}. This gives a canonical identification $\wedge^2_{\Sh{O}_F}\Sh{H}_{\Sh{A}} = \underline{\omega}_{\Sh{A}}\otimes_{\Sh{O}_F \otimes \Sh{O}_{M^{\mathfrak{c}}}} (\underline{\omega}_{\Sh{A}}^{\vee}\otimes \mathfrak{c}) = \Sh{O}_{M^{\mathfrak{c}}} \otimes \mathfrak{cd}^{-1}$, where $\mathfrak{d}$ is the different ideal of $F/\Q$.

Let us now suppose $R$ is a $\Sh{O}_{F, (p)}$-algebra. By abuse of notation, we use the same letters $X, M$ etc for the base-extensions of the above $\Z_{(p)}$-schemes to $R$. The canonical splitting of $\Sh{O}_F \otimes R \simeq R \times R$, where $\Sh{O}_F$ acts on the first component via the embedding $\sigma_1$ and on the second component via $\sigma_2$, induces a decomposition of $\underline{\omega}_{\Sh{A}}$ and $\Sh{H}_{\Sh{A}}$ as
\[
\underline{\omega}_{\Sh{A}} = \omega_{\Sh{A},1} \oplus \omega_{\Sh{A},2}, \quad \Sh{H}_{\Sh{A}} = \Sh{H}_{\Sh{A},1} \oplus \Sh{H}_{\Sh{A},2}.
\]

For $(\underline{k};w) = (k_1, k_2;w) \in \Z^2\times \Z$ such that $k_1 \equiv k_2 \equiv w \text{ mod } 2$, let 
\[
\underline{\omega}^{(\underline{k};w)} := \bigotimes_{1\leq i\leq 2} \left((\wedge^2 \Sh{H}^1_{\Sh{A},i})^{\frac{w-k_i}{2}} \otimes \omega_{\Sh{A},i}^{k_i}\right)
\]

\subsubsection{Descent data}

For any $x \in F^{\times, (p)}_+$ the isomorphism $[x] \colon M^{\mathfrak{c}} \to M^{x\mathfrak{c}}$ sending $(A, \iota, \lambda, \alpha_{K'}) \mapsto (A, \iota, x\lambda, \alpha_{K'})$ induces an isomorphism $[x]^*(\wedge^2_{\Sh{O}_F}\Sh{H}_{\Sh{A}}) \to \wedge^2_{\Sh{O}_F} \Sh{H}_{\Sh{A}}$ such that the following diagram commutes.
\[\begin{tikzcd}
	{[x]^*(\wedge^2_{\Sh{O}_F}\Sh{H}_{\Sh{A}})} & {\wedge^2_{\Sh{O}_F}\Sh{H}_{\Sh{A}}} \\
	{\Sh{O}_{M^{\mathfrak{c}}}\otimes x\mathfrak{cd}^{-1}} & {\Sh{O}_{M^{\mathfrak{c}}}\otimes \mathfrak{cd}^{-1}}
	\arrow[from=1-1, to=1-2]
	\arrow["\simeq", from=1-1, to=2-1]
	\arrow["\simeq", from=1-2, to=2-2]
	\arrow["{x^{-1}}", from=2-1, to=2-2]
\end{tikzcd}\]

The isomorphism $\Sh{O}_{M^{\mathfrak{c}}}\otimes \mathfrak{cd}^{-1} \simeq \wedge^2_{\Sh{O}_F} \Sh{H}_{\Sh{A}}$ defines a canonical generator $\eta_{\can,\lambda}$ of $\wedge^2_{\Sh{O}_F} \Sh{H}_{\Sh{A}}$ given by the image of $1$. Thus we get a natural isomorphism of sheaves over $M$
\begin{align*}
\underline{\omega}^{(\underline{k};w)} &\xrightarrow{\simeq} \underline{\omega}^{\underline{k}} := \bigotimes_{1\leq i\leq 2} \omega_{\Sh{A},i}^{k_i} \\
f &\mapsto \left[\tilde{f}\colon (A,\iota, \lambda, \alpha_{K'}, \omega) \mapsto f(A,\iota, \lambda, \alpha_{K'}, \omega, \eta_{\can,\lambda})\right].
\end{align*}
Now since $\eta_{\can, x\lambda} = x^{-1}\eta_{\can,\lambda}$, we have $\tilde{f}(A,\iota, x\lambda, \alpha_{K'},\omega) = \prod_{1\leq i\leq 2}x^{\frac{w-k_i}{2}}\tilde{f}(A, \iota, \lambda, \alpha_{K'}, \omega)$. Therefore $\underline{\omega}^{(\underline{k};w)}$ as a sheaf over $X$ can also be viewed as a descent of $\underline{\omega}^{\underline{k}}$ with respect to this descent datum. 

\subsection{Hodge height and Hasse invariants on $X$}

We now suppose that $R$ is the ring of integers of a finite extension of $\Q_p$. 
Let $\wp \subset R$ be the maximal ideal of $R$. Let $\mathfrak{X}, \mathfrak{M}$ be the $\wp$-adic completion of $X, M$, and let $\Sh{X}, \Sh{M}$ be the associated rigid analytic spaces. We use this general convention of letter styles for denoting formal schemes, and rigid spaces henceforth. 

\subsubsection{Hasse invariants} Over $R/\wp$ the $\pp_1$-Verschiebung $V_{\pp_1} \colon \Sh{A}_{R/\wp}^{(\pp_1)} \to \Sh{A}_{R/\wp}$ defines the partial Hasse invariant at $\pp_1$, 
\[
 \Ha_1 \in H^0(X_{R/\wp}, \omega_{\Sh{A},1}^{p-1}).
\]

\begin{defn}
    Given a point $x \in \Sh{X}$, define the $\pp_1$-partial Hodge height of $x$ to be 
    \[
    \Hdg_1(x) = v_p(\widetilde{\Ha}_1(x)) \in \Q \cap [0,1]
    \]
    where $\widetilde{\Ha}_1$ is a local lift of $\Ha_1$ around $x$. 
\end{defn}

For $v \in \Q \cap [0,1]$, define $\Sh{X}_v$ to be the inverse image $\Hdg_1^{-1}([0,v]) \subset \Sh{X}$. Let $\mathfrak{X}_v$ be its admissible formal model obtained by blowing up $\mathfrak{X}$ along the ideal locally defined by $(\widetilde{\Ha}_1, p^v)$, and then taking normalisation of the largest open subscheme where the ideal $(\widetilde{\Ha}_1, p^v)$ is generated by $\widetilde{\Ha}_1$.

\subsubsection{Partial canonical subgroups} 
The $p$-divisible group $\Sh{A}[\pp_1^{\infty}]$ is of height 2. We will recall some key facts and properties of canonical subgroups associated with this $p$-divisible group.

\begin{theorem}
 [see {\cite[Appendice A]{Andreatta2018leHS}}]
 \label{T103}
   Let $n \ge 1$. If $v < \frac{1}{p^{n+1}}$, the $p$-divisible group $\Sh{A}[\pp_1^{\infty}]$ over $\mathfrak{X}_v$ admits a canonical subgroup $H^{\can}_n$ of level $n$ that lifts $\ker F_{\pp_1}^n := \ker F^n \cap \Sh{A}[\pp_1^n]$ modulo ${p} / {\underline{\Ha}_1^{\frac{p^n-1}{p-1}}}$, where $\underline{\Ha}_1$ is the ideal generated by local lifts of $\Ha_1$ in $\Sh{O}_{\mathfrak{X}_v}$. Moreover, $\Hdg_1(\Sh{A}/H^{\can}_1) = p\Hdg_1(\Sh{A})$. 
\end{theorem}

\begin{rem}
    In fact for the canonical subgroup of level $1$, Katz \cite{katzpadic} shows that a weaker bound $\Hdg_1 < p/(p+1)$ is sufficient for its existence.
\end{rem}

\subsection{Integral structures on modular and de Rham sheaves}

 Let $n \geq 1$, and $v < \frac{1}{p^{n+1}}$, so that there is a canonical subgroup $H^{\can}_n$ over $\mathfrak{X}_v$. Let $\mathfrak{M}_v$ be the pullback of $\mathfrak{X}_v$ to $\mathfrak{M}$, and $\Sh{M}_v$ its generic fibre. Let $\Sh{IG}_{v, n} \to \Sh{M}_v$ be the partial Igusa tower representing the functor of isomorphisms $\Z/p^n\Z \cong H^{\can, D}_n$ of finite \'{e}tale group schemes, and let $\mathfrak{IG}_{v, n}$ be the normalisation of $\mathfrak{M}_v$ in $\Sh{IG}_{v, n}$. We note that $\Sh{IG}_{v, n}$ is a Galois cover of $\Sh{M}_v$ with Galois group $(\Z/p^n\Z)^{\times}$.

The Hodge--Tate map and its linearisation induces a map 
 \[
  \mathrm{HT}_v \colon H^{\can,D}_n(\mathfrak{IG}_{v, n}) \otimes \Sh{O}_{\mathfrak{IG}_{v, n}}/p^{n-v\frac{p^n-1}{p-1}} 
  \xrightarrow{} \omega_{H^{\can}_n}/p^{n-v\frac{p^n-1}{p-1}} 
  \simeq \omega_{\Sh{A},1}/p^{n-v\frac{p^n-1}{p-1}}
 \]
 whose cokernel is killed by $p^{\frac{v}{p-1}}$ \cite[\S5.5.1]{BoxerPilloni_ModularCurve}. Using this, one can define an integral structure $\omega^{\sharp}_{\Sh{A},1}$ on $\omega_{\Sh{A},1}$ whose existence we state as a theorem.
\begin{theorem}
    There exists an $\Sh{O}_{\mathfrak{IG}_{v, n}}$-sheaf $\omega^{\sharp}_{\Sh{A},1}$, with $p\omega_{\Sh{A},1} \subset \omega^{\sharp}_{\Sh{A},1} \subset \omega_{\Sh{A},1}$, such that $\omega^{\sharp}_{\Sh{A},1}/p^{n-v\frac{p^n}{p-1}}$ is equipped with a marked section coming from the Hodge--Tate map:
    \[
    \mathrm{HT}_v \colon H^{\can, D}_n(\mathfrak{IG}_{v, n}) \otimes \Sh{O}_{\mathfrak{IG}_{v, n}}/p^{n-v\frac{p^n}{p-1}} \xrightarrow{\simeq} \omega^{\sharp}_{\Sh{A},1}/p^{n-v\frac{p^n}{p-1}}.
    \]
\end{theorem}
\begin{proof}
    See \cite[\S5.5.1]{BoxerPilloni_ModularCurve} for a proof in the modular curve case, but the proof is exactly same.
\end{proof}

\begin{cor}
    The isomorphism class of $\omega^{\sharp}_{\Sh{A},1}$ as a line bundle with a marked section defines a class in $\check{H}^1(\mathfrak{IG}_{v, n}, 1+p^{n-v\frac{p^n}{p-1}}\Ga)$.
\end{cor}

In \cite{kazi2024twistedtripleproductpadic}, the first author defined an integral structure on the relative de Rham sheaf $\Sh{H}_{\Sh{A}}$ similar to what was done in \cite{andreatta2021triple} for the case of modular curves. In our present setting we only need to impose an integral structure on $\Sh{H}_{\Sh{A},1}$ since we are only interested in the interpolation sheaf in one variable. This situation is exactly similar to the modular curve case. Thus we define 
\[
\Sh{H}^{\sharp}_{\Sh{A},1} := \omega^{\sharp}_{\Sh{A},1} + \underline{\Ha}_1^{\frac{p}{p-1}}\Sh{H}_{\Sh{A},1}.
\]
We use the prime-to-$p$ polarisation on each connected component of $\mathfrak{M}$ to identify (non-canonically) $\omega_{\Sh{A},1}^{\vee} \simeq \omega_{\Sh{A}^{\vee},1}^{\vee}$. 
Thus we get a Hodge filtration on $\Sh{H}^{\sharp}_{\Sh{A},1}$:
\[
0 \to \omega^{\sharp}_{\Sh{A},1} \to \Sh{H}^{\sharp}_{\Sh{A},1} \to \underline{\Ha}_1^{\frac{p}{p-1}}\omega^{\vee}_{\Sh{A},1} \to 0.
\]
\begin{cor}
    The isomorphism class of $\Sh{H}^{\sharp}_{\Sh{A},1}$ as a vector bundle with a marked section in the sense of \emph{\cite{andreatta2021triple}} defines a class in $\check{H}^1(\mathfrak{IG}_{v, n}, \Sh{B}^{n-v\frac{p^n}{p-1},0}_2)$, where for any $r, r' \in \Q_{>0}$,
    \[
    B_2^{r,r'} := \begin{pmatrix}
        1+p^r\Ga & p^{r'}\Ga \\
        0 & \Gm
    \end{pmatrix}.
    \]
\end{cor}

\begin{lemma}\label{L109}
    Let $A, A'$ be abelian varieties that define points in $\mathfrak{IG}_{v, n}(R)$. Let  $\check{\pi} \colon {A} \to {A}'$ be a cyclic $\pp_1$-isogeny such that $\ker \check{\pi} \cap H^{\can}_1({A}) = 0$. Then the dual isogeny $\pi \colon A' \to A$ identifies $A'/H^{\can}_1(A') \simeq A$, and $\check{\pi}$ is a lift of the Verschiebung of $A'$. Moreover, $\check{\pi}^* \colon \Sh{H}_{{A}',1} \to \Sh{H}_{{A},1}$ restricts to a well-defined map $\check{\pi}^* \colon \Sh{H}^{\sharp}_{{A}',1} \to \Sh{H}^{\sharp}_{{A},1}$ such that the induced map on the graded pieces of the Hodge filtration can be described as follows:
    \[\begin{tikzcd}
	0 & {\omega^{\sharp}_{A',1}} & {\Sh{H}^{\sharp}_{A',1}} & {\underline{\Ha}^{\frac{p}{p-1}}_1(A')\omega^{\vee}_{A',1}} & 0 \\
	0 & {\omega^{\sharp}_{A,1}} & {\Sh{H}^{\sharp}_{A,1}} & {\underline{\Ha}^{\frac{p^2}{p-1}}_1(A')\omega^{\vee}_{A,1}} & 0
	\arrow[from=1-1, to=1-2]
    \arrow[from=2-1, to=2-2]
	\arrow[from=1-2, to=1-3]
	\arrow["\simeq", from=1-2, to=2-2]
	\arrow[from=1-3, to=1-4]
	\arrow["{\check{\pi}^*}", from=1-3, to=2-3]
	\arrow[from=1-4, to=1-5]
	\arrow["{\frac{p}{\underline{\Ha}^{p+1}_1}}", from=1-4, to=2-4]
	\arrow[from=2-2, to=2-3]
	\arrow[from=2-3, to=2-4]
	\arrow[from=2-4, to=2-5]
\end{tikzcd}\]
\end{lemma}

\begin{proof}
    The first point is clear. Therefore by Theorem \ref{T103}, $\underline{\Ha}_1(A) = \underline{\Ha}_1^p(A')$. This explains the second row of the diagram. The isogeny $\check{\pi}$ induces a generic isomorphism of canonical subgroups $H^{\can}_n(A) \simeq H^{\can}_n(A')$, hence an isomorphism on $R$-points. Therefore by the functoriality of the Hodge--Tate map, we get an isomorphism $\omega^{\sharp}_{A',1} \simeq \omega^{\sharp}_{A,1}$. The proof that $\check{\pi}^*$ restricts to a well-defined map $\Sh{H}^{\sharp}_{A',1} \to \Sh{H}^{\sharp}_{A,1}$ has been explained in details in \cite[Prop. 4.12, 4.16]{kazi2024twistedtripleproductpadic}. Let's explain why the induced map on the first graded piece $\underline{\Ha}_1^{\frac{p}{p-1}}(A')\omega^{\vee}_{A',1} \to \underline{\Ha}_1^{\frac{p^2}{p-1}}(A')\omega^{\vee}_{A,1}$ can be identified with multiplication by $\frac{p}{\underline{\Ha}_1^{p+1}}$ upon choosing local trivialisation of the sheaves. The map on the Lie algebras can be identified with the map $(\pi^*)^{\vee} \colon \omega^{\vee}_{A',1} \to \omega^{\vee}_{A,1}$, which is given by multiplication by $\frac{p}{\underline{\Ha}_1}$ since $\check{\pi}^*\circ \pi^* = p$ and $\check{\pi}^* = \underline{\Ha}_1$. The lemma follows after adjusting for the scaling terms.
\end{proof}

\begin{rem}\label{R2010}
    Even though we don't need it here, we make the observation that modulo $\frac{p}{\underline{\Ha}_1^{p+1}}$ the above lemma gives a canonical projection $\Sh{H}^{\sharp}_{A',1}/(\frac{p}{\underline{\Ha}_1^{p+1}}) \to \omega^{\sharp}_{A',1}/(\frac{p}{\underline{\Ha}_1^{p+1}})$ which is a retration to the inclusion $\omega^{\sharp}_{A',1} \to \Sh{H}^{\sharp}_{A',1}$. Moreover, since the kernel of this projection is given by the kernel of the map induced by (the lift of) the Verschiebung of $A'$, this can be identified with the unit root splitting over the ordinary locus. In fact using higher powers of the lift of Verschiebung, say $V^n$, we get a splitting modulo $\frac{p^n}{\underline{\Ha}_1^{\frac{(p+1)(p^n-1)}{p-1}}}$. Therefore, the isomorphism class of $\Sh{H}^{\sharp}_{A',1}$ can be viewed as a class 
    \[
    [\Sh{H}^{\sharp}_{A',1}] \in \check{H}^1\left(\mathfrak{IG}_{v, n}, \begin{pmatrix}
        1+p^{n-v\frac{p^n}{p-1}}\Ga & p^{n-v\frac{(p+1)(p^n-1)}{p-1}}\Ga \\
        0 & \Gm
    \end{pmatrix}\right).
    \]
    This refined data is useful for proving that the $p$-adic iterates of the Gauss-Manin connection on appropriately defined nearly overconvergent modular sheaves converge. See \cite{graham2023padic} and \cite{kazi2024twistedtripleproductpadic} for more details.
\end{rem}

\subsection{Interpolation sheaves for overconvergent and nearly overconvergent forms}
Let $\Lambda = R\llbracket \Z_p^{\times} \rrbracket$, and $\Sh{W} = \Sp\Lambda[1/p]$ be the rigid analytic weight space in one variable. Let $\Sh{W}_r \subset \Sh{W}$ be an open affinoid, with a formal model given by $\mathfrak{W}_r$ such that the universal weight $\kappa \colon \Z_p^{\times} \to \Sh{O}_{\mathfrak{W}_r}^{\times}$ is analytic on $1+p^r\Z_p$, for a fixed $r \in \Q_{>0}$. For this fixed $r$, we choose $v$ small enough and $n$ large enough such that $r \leq n - v\frac{p^n}{p-1}$. By analyticity, the universal weight extends to a character $\kappa \colon \Z_p^{\times}(1+p^r\Ga) \to \Sh{O}_{\mathfrak{W}_r}^{\times}$.

\begin{lemma}
    Given any vector bundle $\Sh{V}$ of rank $2$  on a scheme $Y$ with an increasing filtration by sub-vector bundles constituting a full flag, the map $\check{H}^1(Y, \Sh{B}_2) \to \check{H}^1(Y, \Sh{B}_{n+1})$, where $\Sh{B}_i$ is the upper triangular Borel in $\GL_{i}$, sending $[\Sh{V}] \mapsto [\Sym^n\Sh{V}]$ is induced by a group scheme homomorphism $\Sh{B}_2 \to \Sh{B}_{n+1}$ which can be described as follows:
    \[
    \begin{pmatrix}
        1 & b \\
        0 & 1 
    \end{pmatrix} \mapsto \begin{pmatrix}
        1 & b & b^2 &\cdots & b^n \\
        0 & 1 & 2b & \cdots & nb^{n-1} \\
        \vdots & \vdots & \vdots & \ddots &\vdots \\
        0 & 0 & 0 & \cdots & 1
    \end{pmatrix}, \quad \quad \begin{pmatrix}
        a & 0 \\
        0 & c
    \end{pmatrix} \mapsto \begin{pmatrix}
        a^n & 0 & \cdots & 0 \\
        0 & a^n(\frac{c}{a}) & \cdots & 0 \\
        \vdots & \vdots & \ddots & \vdots \\
        0 & 0 & \cdots & a^n(\frac{c}{a})^n
    \end{pmatrix}.
    \]
    In the first association, the $i$-th column of the image matrix are given by the monomials of the binomial expansion $(b+1)^i$. Moreover, the projection $\Sh{B}_m \to \Sh{B}_{m'}$ for any $m'<m$ given by projecting onto the top left $m' \times m'$ square matrix corresponds to sending a rank $m$ vector bundle with an increasing filtration $\{\Fil_i\}_{i=0}^{m-1}$ to the filtered vector bundle $\Fil_{m'-1}$.
\end{lemma}

\begin{proof}
    Left to the reader.
\end{proof}

The above lemma suggests that for the $r$-analytic universal weight $\kappa$, we can consider the image of $[\Sh{H}^{\sharp}_{\Sh{A},1}]$  in $\check{H}^1(\mathfrak{IG}_{v, n}, \varprojlim_m \Sh{B}_m)$ under the compatible system of group homomorphisms $\Sh{B}_2^{n-v\frac{p^n}{p-1},0} \to \Sh{B}_{m+1}$ which can be described as follows: on the unipotent part, the homomorphism is given by the same homomorphism as in the lemma above. On the diagonal torus $\mathrm{diag}[a,c] \mapsto \mathrm{diag}[a^{\kappa}, a^{\kappa}(\frac{c}{a}), \cdots, a^{\kappa}(\frac{c}{a})^m]$. We note that this makes sense since $r < n-v\frac{p^n}{p-1}$ and $\kappa$ extends to a character on $1+p^r\Ga$. Let us call this homomorphism, as well as the induced map on cohomology $\Sym^{\kappa} \colon \check{H}^1(\mathfrak{IG}_{v, n}, \Sh{B}_2^{n-v\frac{p^n}{p-1}, 0}) \to \check{H}^1(\mathfrak{IG}_{v, n}, \varprojlim_m\Sh{B}_m)$.

\begin{defn}
    Define the Banach sheaf $\Sh{N}^{(\kappa,0;0)}_{\mathfrak{IG}_{v, n}}$ of nearly overconvergent modular forms of weight $(\kappa,0;0)$ on $\mathfrak{IG}_{v, n}$ to be the sheaf whose isomorphism class as a $p$-adically complete filtered colimit of locally free sheaves $\{\Fil_i\Sh{N}^{(\kappa,0;0)}_{\mathfrak{IG}_{v, n}}\}_{i\geq 0}$ is given by $\Sym^{\kappa}([\Sh{H}^{\sharp}_{\Sh{A},1}])$. Define the overconvergent sheaf of weight $\kappa$ to be $\Omega^{(\kappa,0;0)}_{\mathfrak{IG}_{v, n}} := \Fil_0\Sh{N}^{(\kappa,0;0)}_{\mathfrak{IG}_{v, n}}$.
\end{defn}

\subsubsection{An explicit model of the sheaves}

 Though we have defined the sheaves $\Sh{N}^{(\kappa, 0;0)}_{\mathfrak{IG}_{v, n}}, \Omega^{(\kappa,0;0)}_{\mathfrak{IG}_{v, n}}$ only upto isomorphism, canonical representatives for their isomorphism class can be constructed using appropriate geometric vector bundles associated with the sheaves $\Sh{H}^{\sharp}_{\Sh{A},1}$ and $\omega^{\sharp}_{\Sh{A},1}$ as done in \cite{andreatta2021triple}. Our definition above is an algebraic reformulation of the geometric formalism of vector bundle with marked sections of loc. cit. In the following we give a brief account of the geometric construction.

 Let $s = HT_v(P^{\mathrm{univ}}) \in H^{\can,D}_n(\mathfrak{IG}_{v, n})$ be the marked section of the line bundle $\omega^{\sharp}_{\Sh{A},1}$. Let $\mathbb{V}(\Sh{H}^{\sharp}_{\Sh{A},1}) \to \mathfrak{IG}_{v, n}$ be the geometric vector bundle whose sections are linear functionals on $\Sh{H}^{\sharp}_{\Sh{A},1}$, i.e. \[\mathbb{V}(\Sh{H}^{\sharp}_{\Sh{A},1})(\gamma \colon \Spf R' \to \mathfrak{IG}_{v, n}) = \Hom_{R'}(\gamma^*\Sh{H}^{\sharp}_{\Sh{A},1}, R').\] 
Let $\V{\Sh{H}^{\sharp}_{\Sh{A},1}}{s}$ be the subfunctor of $\mathbb{V}(\Sh{H}^{\sharp}_{\Sh{A},1})$ defined by 
\[
\V{\Sh{H}^{\sharp}_{\Sh{A},1}}{s}(\Spf R') = \{f \in \mathbb{V}(\Sh{H}^{\sharp}_{\Sh{A},1})(\Spf R')\,|\, (f \text{ mod }p^{n-v\frac{p^n}{p-1}})(s) = 1\}.
\]
Let $\mathbb{V}({\Os})$ and $\V{\Os}{s}$ be defined analogously. One sees that there is a natural action of the formal group $1+p^{n-v\frac{p^n}{p-1}}\Ga$ on both $\V{\Hs}{s}$ and $\V{\Os}{s}$, where an element $t \in 1+p^{n-v\frac{p^n}{p-1}}\Ga$ acts by sending a linear functional $f$ to $tf$.
\begin{prop}
    $\V{\Hs}{s}$ and $\V{\Os}{s}$ are reprsentable by geometric vector bundles obtained as admissible blow-ups of $\mathbb{V}(\Hs)$ and $\mathbb{V}(\Os)$ respectively. Over a Zariski open $\Spf R' \subset \mathfrak{IG}_{v, n}$, where $\Hs$ and $\Os$ admit trivialisation with a basis compatible with the Hodge filtration $X,Y$, such that $X$ is a lift of the marked section, we have fibre diagrams as follows.
    \begin{equation}\begin{tikzcd}[column sep=huge]
	{\mathfrak{IG}_{v, n}} & {\mathbb{V}(\Hs)} & {\V{\Hs}{s}} \\
	{\Spf R'} & {\Spf R'\langle X,Y \rangle} & {\Spf R'\langle Z,Y \rangle}
	\arrow[from=1-2, to=1-1]
	\arrow[from=1-3, to=1-2]
	\arrow[from=2-1, to=1-1]
	\arrow[from=2-2, to=1-2]
	\arrow[from=2-2, to=2-1]
	\arrow[from=2-3, to=1-3]
	\arrow["{X \mapsto 1+p^{n-v\frac{p^n}{p-1}}Z}"', from=2-3, to=2-2]
\end{tikzcd}\end{equation}
    \begin{equation}\begin{tikzcd}[column sep=huge]
	{\mathfrak{IG}_{v, n}} & {\mathbb{V}(\Os)} & {\V{\Os}{s}} \\
	{\Spf R'} & {\Spf R'\langle X \rangle} & {\Spf R'\langle Z \rangle}
	\arrow[from=1-2, to=1-1]
	\arrow[from=1-3, to=1-2]
	\arrow[from=2-1, to=1-1]
	\arrow[from=2-2, to=1-2]
	\arrow[from=2-2, to=2-1]
	\arrow[from=2-3, to=1-3]
	\arrow["{X \mapsto 1+p^{n-v\frac{p^n}{p-1}}Z}"', from=2-3, to=2-2]
\end{tikzcd}\end{equation}
\end{prop}

\begin{proof}
    See \cite[31]{kazi2024twistedtripleproductpadic}.
\end{proof}

\begin{prop}\label{P1015}
    Let $\rho\colon \V{\Hs}{s} \to \mathfrak{IG}_{v, n}$ be the map as above. Let $\Sh{N}^{(\kappa,0;0)}_{\mathfrak{IG}_{v, n}} := \rho_*\Sh{O}_{\V{\Hs}{s}}[\kappa]$ be the subsheaf of $\rho_*\Sh{O}_{\V{\Hs}{s}}$ on which the formal group $1+p^{n-v\frac{p^n}{p-1}}\Ga$ acts via the character $\kappa$. Then $\Sh{N}^{(\kappa,0;0)}_{\mathfrak{IG}_{v, n}}$ is a Banach sheaf representing the class $\Sym^{\kappa}([\Hs])$ as above. Similarly, letting $\nu \colon \V{\Os}{s} \to \mathfrak{IG}_{v, n}$ be the structure morphism as above, the sheaf $\Omega^{(\kappa,0;0)}_{\mathfrak{IG}_{v, n}} := \nu_*\Sh{O}_{\V{\Os}{s}}[\kappa]$ represents the class in $\check{H}^1(\mathfrak{IG}_{v, n}, \Gm)$ of the image of $[\Os]$ under the map $\Sym^{\kappa} \colon \check{H}^1(\mathfrak{IG}_{v, n}, 1+p^{n-v\frac{p^n}{p-1}}\Ga) \xrightarrow{\kappa} \check{H}^1(\mathfrak{IG}_{v, n}, \Gm)$.

    Moreover, locally on a Zariski open $\Spf R'$ as above where $\Hs$ and $\Os$ both admit trivialisation, we have
    \[
    \Sh{N}^{(\kappa,0;0)}_{\mathfrak{IG}_{v, n}|\Spf R'} = R'\left\langle \frac{Y}{1+p^{n-v\frac{p^n}{p-1}}Z}\right\rangle \cdot (1+p^{n-v\frac{p^n}{p-1}}Z)^{\kappa}, \quad \Omega^{(\kappa,0;0)}_{\mathfrak{IG}_{v, n}|\Spf R'} = R'\cdot (1+p^{n-v\frac{p^n}{p-1}}Z)^{\kappa}.
    \]
\end{prop}

\begin{proof}
    The first part follows because the algebraic definition is just a reformulation of the geometric definition. The second part is proved in \cite[Corollary 4.37]{kazi2024twistedtripleproductpadic}.
\end{proof}

\subsection{Comparison with classical sheaves}

\begin{cor}
    Over the generic fibre $\Sh{IG}_{v, n}$, the specialisation of $\Fil_i\Sh{N}^{(\kappa,0;0)}_{\mathfrak{IG}_{v, n}}$ to the weight $\kappa = i$ is canonically isomorphic to $\Sym^i\Sh{H}_{\Sh{A},1}$.
\end{cor}

\begin{proof}
    It is clear that the specialisation of $\kappa$ to $i$ gives a canonical identification $\Fil_i\Sh{N}^{(\kappa=i,0;0)}_{\mathfrak{IG}_{v, n}} \simeq \Sym^i\Hs$. But this is naturally isomorphic to $\Sym^i\Sh{H}_{\Sh{A},1}$ since $\Hs = \Sh{H}_{\Sh{A},1}$ over the generic fibre.
\end{proof}

We now descend the sheaves just constructed to the formal scheme $\mathfrak{M}_v$. Since $\mathfrak{IG}_{v, n}$ is the normalisation of $\mathfrak{M}_v$ in $\Sh{IG}_{v, n}$, the Galois action of $\Z_p^{\times}$ acting through the finite quotient $(\Z/p^n\Z)^{\times}$ on $\Sh{IG}_{v, n} \to \Sh{M}_v$ extends to an action on $f\colon \mathfrak{IG}_{v, n} \to \mathfrak{M}_v$. Moreover, this action induces an action on both $\omega^{\sharp}_{\Sh{A},1}$ and $\Sh{H}^{\sharp}_{\Sh{A},1}$ which can be described as follows. For an element $t \in \Z_p^{\times}$, the map $[t] \colon \mathfrak{IG}_{v, n} \to \mathfrak{IG}_{v, n}$ sends a generator $P \in H^{\can, D}_n(R') \mapsto \bar{t}^{-1}P$ for any $R$-algebra $R'$ over the generic fibre. Here $\bar{t}$ is the class of $t$ in $(\Z/p^n\Z)^{\times}$. We have a canonical identification given by the identity map $[t]^*\Sh{H}^{\sharp}_{\Sh{A},1} \xrightarrow{\id} \Sh{H}^{\sharp}_{\Sh{A},1}$, which sends the marked section $[t]^*HT_v(P) \mapsto \bar{t}^{-1}HT_v(P)$. Define the action of $t \in \Z_p^{\times}$ on $\Sh{H}^{\sharp}_{\Sh{A},1}$ as: 
\[
\Sh{H}^{\sharp}_{\Sh{A},1} \xrightarrow{[t]^*} [t]^*\Sh{H}^{\sharp}_{\Sh{A},1} \xrightarrow{t\cdot \id} \Sh{H}^{\sharp}_{\Sh{A},1}.\]
Note that this map sends the marked section of $[t]^*\Sh{H}^{\sharp}_{\Sh{A},1}$ to the marked section of $\Sh{H}^{\sharp}_{\Sh{A},1}$. Hence by the formalism of vector bundles with marked sections \cite[Definition 2.3]{andreatta2021triple}, we get a map $[t]^*\Sh{N}^{(\kappa, 0;0)}_{\mathfrak{IG}_{v, n}} \to \Sh{N}^{(\kappa, 0;0)}_{\mathfrak{IG}_{v, n}}$ which respects the filtration. This defines the descent datum on $\Sh{N}^{(\kappa,0;0)}_{\mathfrak{IG}_{v, n}}$.

\begin{defn}
    Define the sheaf $\Sh{N}^{(\kappa,0;0)}_{\mathfrak{M}_v}$ on $\mathfrak{M}_v$ as $f_*\Sh{N}^{(\kappa,0;0)}_{\mathfrak{IG}_{v, n}}[\kappa]$, i.e. the subsheaf of $f_*\Sh{N}^{(\kappa,0;0)}_{\mathfrak{IG}_{v, n}}$ on which $\Z_p^{\times}$ acts via the character $\kappa$. Define $\Omega^{(\kappa,0;0)}_{\mathfrak{M}_v}$ to be $\Fil_0\Sh{N}^{(k,0;0)}_{\mathfrak{M}_v}$ (which is well-defined because the $\Z_p^{\times}$ action respects the filtration on $\Sh{N}^{(\kappa,0;0)}_{\mathfrak{IG}_{v, n}}$.)
\end{defn}

\begin{cor}
    The sheaf $\Omega^{(\kappa,0;0)}_{\Sh{M}_v}$ is a line bundle, and $\Sh{N}^{(\kappa,0;0)}_{\Sh{M}_v}$ is a completed colimit of its filtered pieces which are locally free.
\end{cor}

\begin{proof}
    The claim is easy to see over the generic fibre. Since the map $\Sh{IG}_{v, n} \to \Sh{M}_v$ is finite \'{e}tale, the result follows from the corresponding facts about $\Omega^{(\kappa,0;0)}_{\Sh{IG}_{v, n}}$ and $\Sh{N}^{(\kappa,0;0)}_{\Sh{IG}_{v, n}}$. The stronger statement over the formal schemes requires more involved argument and is proved in \cite[Theorem 4.47]{kazi2024twistedtripleproductpadic}.
\end{proof}

We will now descend the sheaves $\Omega^{(2\kappa,0;0)}_{\mathfrak{M}_v}$ and $\Sh{N}^{(2\kappa,0;0)}_{\mathfrak{M}_v}$ to the Shimura variety $\mathfrak{X}_v$ by describing the action of $F^{\times, (p)}_+$. Consider the map $[x] \colon \mathfrak{M}_v \to \mathfrak{M}_v$ sending $(A,\iota,\lambda,\alpha_{K'}) \mapsto (A, \iota, x\lambda, \alpha_{K'})$. The identity $\id \colon [x]^*\Sh{H}^{\sharp}_{\Sh{A},1} \to \Sh{H}^{\sharp}_{\Sh{A},1}$ sends the marked section to the marked section, and hence induces an isomorphism $\id \colon [x]^*\Sh{N}^{(2\kappa,0;0)}_{\mathfrak{M}_v} \to \Sh{N}^{(2\kappa,0;0)}_{\mathfrak{M}_v}$. This defines an action of $F^{\times, (p)}_+$ on $\Sh{N}^{(2\kappa,0;0)}_{\mathfrak{M}_v}$ which respects the filtration.

\begin{defn}
    Let $g \colon \mathfrak{M}_v \to \mathfrak{X}_v$ be the projection. Define $\Sh{N}^{(2\kappa,0;0)}$ as a sheaf over $\mathfrak{X}_v$ to be $g_*\Sh{N}^{(2\kappa,0;0)}_{\mathfrak{M}_v}[\kappa^{-1}]$, i.e. the subsheaf of $g_*\Sh{N}^{(2\kappa,0;0)}_{\mathfrak{M}_v}$ on which $F^{\times,(p)}_+$ acts via $\kappa^{-1}$. Define $\Omega^{(2\kappa,0;0)}$ to be $\Fil_0\Sh{N}^{(2\kappa,0;0)}$.
\end{defn}

\begin{rem}
    The sheaf $\Sh{N}^{(2\kappa,0;0)}$ over the generic fibre $\Sh{X}$ is thus a \emph{small, locally projective} Banach sheaf in the sense of \cite[Definition 2.5.10]{boxer2021highercolemantheory}.
\end{rem}


\section{Theory at level $X_0(\pp_1)$}\label{S3}

 We have so far worked on a Shimura variety of spherical level at $p$; but in order to understand Hecke operators, we shall need to compare with classical modular forms with non-trivial level at $\pp_1$.

 \subsection{Shimura varieties and models}

  For $K^{(p)}$ as above, we define
  \[ 
    K_0(\pp_1) \coloneqq K^{(p)} \cdot \operatorname{Iw}_{\pp_1} \cdot \GL_2(\Sh{O}_{F, \pp_2}), \quad\text{where}\quad 
    \operatorname{Iw}_{\pp_1} = \left\{ g \in \GL_2(\Sh{O}_{F,\pp_1}) \, | \, g \equiv \left(\begin{smallmatrix}
     \ast & \ast \\
     0 & \ast 
     \end{smallmatrix}\right) \text{ mod } \pp_1\right\}.
   \]
   We define Shimura varieties $X_0(\pp_1)$, $M_0(\pp_1)$, $M_0^{\mathfrak{c}}(\pp_1)$ analogously to the $X, M, M^{\mathfrak{c}}$ above, so $M_0(\pp_1)$ is a moduli space for $(A, \iota, \lambda, \alpha_{K^{(p)}})$ as before with the additional data of a finite flat degree $p$ subgroup-scheme $H \subset A[\pp_1]$. We can (and do) choose our compactification data so that there is a natural map $X_0(\pp_1)_{\Q} \to X_\Q$ extending the natural quotient map on the open Shimura varieties.

   These schemes are not smooth, but they are normal and relative local complete intersection over $\Spec \Z_{(p)}$, with smooth fibre over $\Q$ \cite{Pappas}.

\subsection{Fargues' degree on $X_0(\pp_1)$}

In \cite{Fargues10}, Fargues defines the degree of a finite flat commutative group scheme. Consider the degree function defined on $\Sh{X}_0(\pp_1)$ which sends a point $x = (A, \iota, \lambda, \alpha_{K^{(p)}}, H)$ to $\deg H \in [0,1]$. For any rational interval $[a,b] \subset \Q \cap [0,1]$, let $\Sh{X}_0(\pp_1)_{[a,b]} = \deg^{-1}([a,b])$. Then $\Sh{X}_0(\pp_1)_{[a,b]}$ is admissible open and quasi-compact. Let us denote by $\Sh{X}_0(\pp_1)_{[0, a[\, \cup\, ]b,1]}$ the complement $\Sh{X}_0(\pp_1)\setminus \Sh{X}_0(\pp_1)_{[a,b]}$. 

\begin{rem}
    In this parametrisation, the two extremal points $0, 1$ correspond to the locus where the universal subgroup $H$ is \'{e}tale and multiplicative respectively.
\end{rem}

\begin{theorem}\label{T105}
We have the following results relating the canonical subgroup and the degree function on $\Sh{X}_0(\pp_1):$
\begin{enumerate}
    \item For any point of $\Sh{X}_0(\pp_1)$ given by an abelian variety $A$ with additional structure, we have $\sum_{H' \subset A[\pp_1]} \deg H' = 1$. Moreover, either all the degrees are equal \emph{(}and hence equal to $1/(p+1)$\emph{)}, or there exists a canonical subgroup $H^{\can}_1$ with $\deg H^{\can}_1 > \tfrac{1}{(p+1)}$, and for all $H' \neq H^{\can}_1$ we have $\deg H' = \frac{1-\deg H^{\can}_1}{p} < \tfrac{1}{(p+1)}$.
    \item If $a < 1/(p+1)$, $\Sh{X}_0(\pp_1)_{[0,a]}$ admits a canonical subgroup $H^{\can}_1 \neq H$, and $\deg H = \Hdg_1/p$. Therefore, the natural map $\Sh{X}_0(\pp_1) \to \Sh{X}$ induces a map $\Sh{X}_0(\pp_1)_{[0,a]} \to \Sh{X}_{pa}$ in this case.
    \item If $a > 1/(p+1)$, then $\Sh{X}_0(\pp_1)_{[a,1]}$ admits a canonical subgroup and $H^{\can}_1 = H$, and $\deg H = 1 - \Hdg_1$ in this case. Therefore, we get a map $\Sh{X}_0(\pp_1)_{[a,1]} \to \Sh{X}_{1-a}$. 
\end{enumerate}
\end{theorem}
\begin{proof}
    See \cite[Theorem 5.4]{BoxerPilloni_ModularCurve} and \cite[Theorem 6]{Fargues11}.
\end{proof}

For $a<v/p$ or $a>1-v$, we have the maps $\Sh{X}_0(\pp_1)_{[0,a]} \to \Sh{X}_v$ and $\Sh{X}_0(\pp_1)_{[a,1]} \to \Sh{X}_v$ respectively (Theorem \ref{T105}). We can thus pull back $\Sh{N}^{(2\kappa,0;0)}$ and $\Omega^{(2\kappa,0;0)}$ along either of these morphisms to define interpolation sheaves of nearly overconvergent and overconvergent modular forms on these admissible opens inside $\Sh{X}_0(\pp_1)$.

\subsection{The \texorpdfstring{$U_{\pp_1}$}{Up1} operator}
For $a > 1-v$ as above we construct the action of the $U_{\pp_1}$ operator on the cohomology $R\Gamma\left(\Sh{X}_0(\pp_1)_{[a,1]}, \Sh{N}^{(2\kappa,0;0)}\right)$.

\subsubsection{The \texorpdfstring{$U_{\pp_1}$}{Up1} correspondence}
Let $\varpi_1 \in (\Sh{O}_F\otimes \Z_p) = \Z_p \times \Z_p$ be the element that is $p$ in the $\pp_1$ component and $1$ in the other. Let $x_1 \in F^{\times, +}$ be such that $\varpi_1/x_1 \in (\Sh{O}_F \otimes \Z_p)^{\times}$. Let $\Sh{M}_0(\pp_1^2) \subset \Sh{M}_0(\pp_1) \times \Sh{M}_0(\pp_1)$ be (a toroidal compactification) underlying the $U_{\pp_1}$ correspondence such that the projections $p_1, p_2$ are finite flat. Away from the boundary $\Sh{M}_0(\pp_1^2)$ parametrises tuples $(A,H,A',H')$ (along with other data that we ignore in the notation) with $\pi \colon A \to A'$ a cyclic $\pp_1$ isogeny such that $\ker \pi \cap H = 0$. We have $p_1(A,H,A',H') = (A,H)$. We have a slightly non-canonical convention for $p_2$. Note that if $(A, \iota, \lambda, \alpha_{K^{(p)}}, H)$ defines a point of $\Sh{M}_0^{\mathfrak{c}}(\pp_1)$, then $(A', \pi_*\iota, \pi_*\lambda, \pi_*\alpha_{K^{(p)}}, H')$ naturally defines a point in $\Sh{M}_0^{\pp_1\mathfrak{c}}(\pp_1)$. However, since in our definition of $\Sh{M}_0(\pp_1)$ we range over polarisation modules prime to $p$, we define $p_2(A, \iota, \lambda, \alpha_{K^{(p)}}, H) = (A', \pi_*\iota, x_1^{-1}\pi_*\lambda, \pi_*\alpha_{K^{(p)}}, H')$ which is naturally a point in $\Sh{M}_0^{x_1^{-1}\pp_1\mathfrak{c}}(\pp_1)$. We also let $\mathfrak{M}_0(\pp_1^2)$ be the normalisation of $\mathfrak{M}_0(\pp_1) \times \mathfrak{M}_0(\pp_1)$ in $\Sh{M}_0(\pp_1^2)$. 

Let $a$ be as above. Consider the isomorphism $\Sh{M}_0(\pp_1)_{[a,1]} \xrightarrow{\sim} \Sh{M}_{1-a}$, and let $\mathfrak{M}_0(\pp_1)_{[a,1]}$ be the normalisation of $\mathfrak{M}_{1-a}$ in $\Sh{M}_0(\pp_1)_{[a,1]}$. We note that the projection $\mathfrak{M}_0(\pp_1)_{[a,1]} \to \mathfrak{M}_{1-a}$ is also an isomorphism. 

\begin{lemma}
    For $a$ as above, $p_2(p_1^{-1}\mathfrak{M}_0(\pp_1)_{[a,1]}) \subset \mathfrak{M}_0(\pp_1)_{[1-\frac{1-a}{p}, 1]}$.
\end{lemma}

\begin{proof}
    This follows from a simple computation using Theorem \ref{T105}. See also \cite[Proposition 5.5]{BoxerPilloni_ModularCurve}.
\end{proof}

In \cite[\S6]{kazi2024twistedtripleproductpadic} we showed that the Hecke correspondence $U_{\pp_1}$ over $\mathfrak{M}_{1-a}$ lifts to a correspondence on $\mathfrak{IG}_{v, n}$. Since the universal isogeny $\check{\pi} \colon p_1^*\Sh{A} \to p_2^*\Sh{A}$ has kernel disjoint from the canonical subgroup, Lemma \ref{L109} implies that $\check{\pi}^* \colon p_2^*\Sh{H}_{\Sh{A},1} \to p_1^*\Sh{H}_{\Sh{A},1}$ restricts to a map $\check{\pi}^* \colon p_2^*\Hs \to p_1^*\Hs$ that sends the marked section to the marked section. Thus, by functoriality it induces a map $\check{\pi}^* \colon p_2^*\Sh{N}^{(2\kappa,0;0)}_{\mathfrak{M}_{1-a}} \to p_1^*\Sh{N}^{(2\kappa,0;0)}_{\mathfrak{M}_{1-a}}$. Pulling back via the isomorphism $\mathfrak{M}_0(\pp_1)_{[a,1]} \simeq \mathfrak{M}_{1-a}$, we get a map $\check{\pi}^* \colon p_2^*\Sh{N}^{(2\kappa,0:0)}_{\mathfrak{M}_0(\pp_1)_{[a,1]}} \to p_1^*\Sh{N}^{(2\kappa,0;0)}_{\mathfrak{M}_0(\pp_1)_{[a,1]}}$.

\begin{defn}
    Let $C_{[a,1]} = p_1^{-1}(\Sh{M}_0(\pp_1)_{[a,1]})$. Define the $U_{\pp_1}$ operator on $R\Gamma(\Sh{M}_0(\pp_1)_{[a,1]}, \Sh{N}^{(2\kappa,0;0)})$ as 
    \[
    \begin{tikzcd}[]
    R\Gamma(\Sh{M}_0(\pp_1)_{[a,1]}, \Sh{N}^{(2\kappa,0;0)}) \ar[r] & R\Gamma(p_2(C_{[a,1]}), \Sh{N}^{(2\kappa,0;0)}) \ar[r, "p_2^*"] \ar[draw = none]{d}[name = X, anchor = center]{} & R\Gamma(C_{[a,1]}, p_2^*\Sh{N}^{(2\kappa,0;0)}) \ar[dl, rounded corners, to path = { -- ([xshift = 2ex]\tikztostart.east) |- (X.center) \tikztonodes -| ([xshift = -2ex]\tikztotarget.west) -- (\tikztotarget)}] & \\
    & R\Gamma(C_{[a,1]}, p_1^*\Sh{N}^{(2\kappa,0;0)}) \ar[r, "{(\star)\cdot{\text{tr}_{p_1}}}"] & R\Gamma(\Sh{M}_0(\pp_1)_{[a,1]}, \Sh{N}^{(2\kappa,0;0)})
\end{tikzcd}
    \]
    where $\star = (\varpi_1/x_1)^{(\kappa,0)}p^{-1}$.
\end{defn}

\begin{cor}
    The $U_{\pp_1}$ operator on $R\Gamma(\Sh{M}_0(\pp_1)_{[a,1]}, \Sh{N}^{(2\kappa,0;0)})$ descends to an operator on \\$R\Gamma(\Sh{X}_0(\pp_1)_{[a,1]}, \Sh{N}^{(2\kappa,0;0)})$.
\end{cor}

\begin{proof}
    This is clear because all the arrows in the definition above commute with the action of $F^{\times, (p)}_+$ on the polarisation.
\end{proof}

\begin{rem}
    In general, for the sheaf $\underline{\omega}^{(k_1,k_2;w)}$ of classical algebraic weight, the $U_{\pp_1}$ operator on $R\Gamma(X_0(\pp_1), \underline{\omega}^{(k_1,k_2;w)})$ is defined similarly using the correspdence described above, and with the normalisation factor $\star = (\varpi_1/x_1)^{(\frac{k_1-w}{2},0)}p^{-1}$. This operator is optimally integral in the sense of \cite{FakhruddinPilloni}. Our definition therefore implies that on $R\Gamma(\Sh{X}_0(\pp_1)_{[a,1]}, \Omega^{(2\kappa,0;0)})$, the $U_{\pp_1}$ operator interpolates the optimally integral $U_{\pp_1}$ operator.
\end{rem}

\begin{cor}\label{C1024}
    The $U_{\pp_1}$ operator on $R\Gamma(\Sh{X}_0(\pp_1)_{[a,1]}, \Sh{N}^{(2\kappa,0;0)})$ satisfies the following properties:
    \begin{enumerate}
        \item It respects the filtration on $\Sh{N}^{(2\kappa,0;0)}$.
        \item The induced map on $R\Gamma(\Sh{X}_0(\pp_1)_{[a,1]}, \Gr_n\Sh{N}^{(2\kappa,0;0)})$ has slope $\geq n\cdot v_p(\frac{p}{\Ha_1^{-(p+1)}})$.
    \end{enumerate} 
\end{cor}

\begin{proof}
    The first point follows from the functorial construction of $\check{\pi}^* \colon p_2^*\Sh{N}^{(2\kappa,0;0)}_{\mathfrak{M}_0(\pp_1)_{[a,1]}} \to p_1^*\Sh{N}^{(2\kappa,0;0)}_{\mathfrak{M}_0(\pp_1)_{[a,1]}}$. The second point follows from the local description of the sheaf as in Proposition \ref{P1015} and Lemma \ref{L109}.
\end{proof}

\begin{prop}
    The $U_{\pp_1}$ operator acts compactly on $R\Gamma(\Sh{X}_0(\pp_1)_{[a,1]}, \Sh{N}^{(2\kappa,0;0)})$.
\end{prop}

\begin{proof}
    Since the sheaf $\Sh{N}^{(2\kappa,0;0)}$ is locally projective (in fact locally orthonormalisable), the second point of Corollary \ref{C1024} reduces us to proving compactness on $R\Gamma(\Sh{X}_0(\pp_1)_{[a,1]}, \Omega^{(2\kappa,0;0)})$. Here the claim follows from \cite[Lemma 2.5.23]{boxer2021highercolemantheory}. We remark that in the statement of the lemma in loc. cit. the sheaf $\Sh{F}$ is defined over the whole space $\Sh{X}$, and the base is a non-archimedean field, but in the proof all that is used is the fact that $\Sh{X}$ is proper, and $\Sh{F}_{|\Sh{U}}$ is the generic fibre of a flat formal Banach sheaf. 
\end{proof}

\begin{defn}
    For $h \in \Q_{>0}$, let $R\Gamma(\Sh{X}_0(\pp_1)_{[a,1]}, \Sh{N}^{(2\kappa+k_1,k_2;w)})^{\leq h}$ be the slope $\leq h$ part of the cohomology for the $U_{\pp_1}$ operator, and let $e^{\leq h}$ be the projector onto this space.
\end{defn}

\begin{cor}
    The natural map \[R\Gamma(\Sh{X}_0(\pp_1)_{[a,1]}, \Fil_n\Sh{N}^{(2\kappa+k_1,k_2;w)})^{\leq h} \to R\Gamma(\Sh{X}_0(\pp_1)_{[a,1]}, \Sh{N}^{(2\kappa+k_1,k_2;w)})^{\leq h}\]
    is an isomorphism for $n \gg 0$. 
\end{cor}

\begin{proof}
    Follows from the same proof as of Corollary \ref{C1024} (2).
\end{proof}

\subsection{The overconvergent projection}
We have a natural inclusion $H^1(\Sh{X}_0(\pp_1)_{[a,1]}, \Omega^{(2\kappa+k_1,k_2;w)}) \to H^1(\Sh{X}_0(\pp_1)_{[a,1]}, \Fil_n\Sh{N}^{(2\kappa+k_1,k_2;w)})$ for any $n \geq 0$. In this section we construct a meromorphic retraction to this map.

In \cite[\S5]{kazi2024twistedtripleproductpadic} following the work of \cite{andreatta2021triple} we showed that the Gauss-Manin connection $\nabla \colon \Sh{H}_{\Sh{A}} \to \Sh{H}_{\Sh{A}} \otimes \Omega_{M}(\log D)$ induces a connection $\nabla \colon \Sh{N}^{(2\kappa+k_1,k_2;w)} \to \Sh{N}^{(2\kappa+k_1,k_2;w)} \hat{\otimes} \Omega_{\Sh{M}_{1-a}}(\log D)$ that is equivariant for the action of $F^{\times, (p)}_{+}$ and therefore descends to a connection on $\Sh{X}_{1-a}$. Here $D$ is the boundary divisor in the toroidal compactification. Using the Kodaira-Spencer isomorphism $\underline{\omega}^{(2,0;0)} \oplus \underline{\omega}^{(0,2;0)} \simeq \Omega_{\Sh{X}_{1-a}}$, we can therefore write $\nabla = \nabla_1 + \nabla_2$ where $\nabla_1 \colon \Sh{N}^{(2\kappa+k_1,k_2;w)} \to \Sh{N}^{(2\kappa+k_1+2,k_2;w)}$ and $\nabla_2 \colon \Sh{N}^{(2\kappa+k_1,k_2;w)} \to \Sh{N}^{(2\kappa+k_1, k_2+2;w)}$ are the two projections. We pullback these maps to $\Sh{X}_0(\pp_1)_{[a,1]}$. We remark that in the notation of loc. cit. Definition 5.7, $\nabla_1$ corresponds to $\nabla_{(2\kappa+k_1,k_2;w)}(\sigma_1)$.

We recall two important properties of $\nabla_1$.
\begin{prop}\label{P1029}
    \begin{enumerate}
        \item $\nabla_1$ satisfies Griffiths' transversality for the filtration on $\Sh{N}^{(2\kappa+k_1,k_2;w)}$, i.e. 
        \[
        \nabla_1\left(\Fil_n\Sh{N}^{(2\kappa+k_1;k_2;w)}\right) \subset \Fil_{n+1}\Sh{N}^{(2\kappa+k_1+2,k_2;w)}.
        \]
        \item The $\Sh{O}_{\Sh{X}_0(\pp_1)_{[a,1]}}$-linear map $\nabla_1 \colon \Gr_n \Sh{N}^{(2\kappa+k_1;k_2;w)} \to \Gr_{n+1}\Sh{N}^{(2\kappa+k_1+2,k_2;w)}$ is an isomorphism onto $(2\kappa + k_1 - n)\Gr_{n+1}\Sh{N}^{(2\kappa+k_1+2,k_2;w)}$, where we abuse notation to view $\kappa$ as an element in $\Sh{O}_{\Sh{W}_r}$ such that $\kappa(t) = \exp(\kappa\log{t})$ for any $t \in 1+p^r\Z_p$. \\
        In particular, using the local coordinates of Proposition \ref{P1015}, and letting $V = \frac{Y}{1+p^{n-v\frac{p^n}{p-1}}Z}$, so that \[\Sh{N}^{(2\kappa+k_1,k_2;w)}(\Sp R) = R\langle V\rangle \cdot (1+p^{n-\frac{p^n}{p-1}}Z)^{(2\kappa+k_1)}\omega_2^{k_2}\] for some $\Sp R \subset \Sh{X}_0(\pp_1)_{[a,1]}$ trivialising $\omega^{\sharp}_{\Sh{A},1}, \Sh{H}^{\sharp}_{\Sh{A},1}$ and $\omega_{\Sh{A},2}$, we have 
        \[
        \nabla_1\left(V^n\cdot (1+p^{n-\frac{p^n}{p-1}Z})^{2\kappa+k_1}\omega_2^{k_2}\right) = s + (2\kappa+k_1 - n)V^{n+1}\cdot (1+p^{n-\frac{p^n}{p-1}Z})^{2\kappa+k_1+2}\omega_2^{k_2}
        \]
        for some $s \in \Fil_n\Sh{N}^{(2\kappa+k_1+2,k_2;w)}$.
    \end{enumerate}
\end{prop}

\begin{proof}
    The first point is proved in \cite[Theorem 5.5]{kazi2024twistedtripleproductpadic}. The second point is Corollary 5.8 of loc. cit.
\end{proof}

We will study the hypercohomology $\mathbf{H}^{\bullet}_{\dR,1}(\Sh{X}_0(\pp_1)_{[a,1]}, \Fil_n\Sh{N}^{(2\kappa+k_1-2,k_2;w)})$ of the complex 
\begin{equation}\label{*}
\Fil_n\Sh{N}^{(2\kappa+k_1-2,k_2;2)} \xrightarrow{\nabla_1} \Fil_{n+1}\Sh{N}^{(2\kappa+k_1,k_2;w)}. \tag{$\ast$}
\end{equation}

\begin{lemma}\label{L1030}
    Let $\lambda = \prod_{i=0}^{n}(2\kappa+k_1-2-i)$. Then for any admissible open affinoid $\Sh{U} \subset \Sh{X}_0(\pp_1)_{[a,1]}\otimes_{\Sh{O}_{\Sh{W}_r}} \Sh{O}_{\Sh{W}_r}[\lambda^{-1}]$, Griffiths' transversality induces an isomorphism of $\Sh{O}_{\Sh{W}_r}[\lambda^{-1}]$-modules
    \[
    \psi^{-1}_{\Sh{U}} \colon \Omega^{(2\kappa+k_1,k_2;w)}(\Sh{U}) \xrightarrow{\sim} \frac{\Fil_{n+1}\Sh{N}^{(2\kappa+k_1,k_2;w)}(\Sh{U})}{\nabla_1(\Fil_n\Sh{N}^{(2\kappa+k_1-2,k_2;w)}(\Sh{U}))}.
    \]
\end{lemma}

\begin{proof}
    This follows from Proposition \ref{P1029} (2).
\end{proof}

Equip the complex (\ref{*}) with the Hodge filtration, i.e. \[F^0(*) = (*),  \qquad F^1({*}) = 0 \to \Fil_{n+1}\Sh{N}^{(2\kappa+k_1,k_2;w)}.\]
Then the Hodge to de Rham spectral sequence induces an exact sequence
\begin{equation}\label{E03}
0 \to \frac{H^1(\Sh{X}_0(\pp_1)_{[a,1]}, \Fil_{n+1}\Sh{N}^{(2\kappa+k_1,k_2;w)})}{\nabla_1\left(H^1(\Sh{X}_0(\pp_1)_{[a,1]}, \Fil_{n}\Sh{N}^{(2\kappa+k_1-2,k_2;w)})\right)} \to \mathbf{H}^{2}_{\dR,1}(\Sh{X}_0(\pp_1)_{[a,1]}, \Fil_n\Sh{N}^{(2\kappa+k_1-2,k_2;w)}).
\end{equation}
We make this more explicit using \v{C}ech resolutions for the sheaves in (\ref{*}). Fix an admissible covering $\Sh{U} = \{\Sh{U}_i\}_{1\leq i\leq n}$ of $\Sh{X}_0(\pp_1)_{[a,1]}$ by admissible open affinoids. We get a double complex as follows.

\[\begin{tikzcd}
	0 & {\Fil_n\Sh{N}^{(2\kappa+k_1-2,k_2;w)}} & {\Fil_{n+1}\Sh{N}^{(2\kappa+k_1,k_2;w)}} & 0 \\
	0 & {\prod_i\Fil_n\Sh{N}^{(2\kappa+k_1-2,k_2;w)}_{|\Sh{U}_i}} & {\prod_i\Fil_{n+1}\Sh{N}^{(2\kappa+k_1,k_2;w)}_{|\Sh{U}_i}} & 0 \\
	0 & {\prod_{i<j}\Fil_n\Sh{N}^{(2\kappa+k_1-2,k_2;w)}_{|\Sh{U}_{ij}}} & {\prod_{i<j}\Fil_{n+1}\Sh{N}^{(2\kappa+k_1,k_2;w)}_{|\Sh{U}_{ij}}} & 0 \\
	0 & {\prod_{i<j<k}\Fil_n\Sh{N}^{(2\kappa+k_1-2,k_2;w)}_{|\Sh{U}_{ijk}}} & {\prod_{i<j<k}\Fil_{n+1}\Sh{N}^{(2\kappa+k_1,k_2;w)}_{|\Sh{U}_{ijk}}} & 0 \\
	 & \vdots & \vdots & 
	\arrow[from=1-1, to=1-2]
	\arrow["{\nabla_1}", from=1-2, to=1-3]
	\arrow["\partial", from=1-2, to=2-2]
	\arrow[from=1-3, to=1-4]
	\arrow["\partial", from=1-3, to=2-3]
	\arrow[from=2-1, to=2-2]
	\arrow["{\nabla_1}", from=2-2, to=2-3]
	\arrow["\partial", from=2-2, to=3-2]
	\arrow[from=2-3, to=2-4]
	\arrow["\partial", from=2-3, to=3-3]
	\arrow[from=3-1, to=3-2]
	\arrow["{\nabla_1}", from=3-2, to=3-3]
	\arrow["\partial", from=3-2, to=4-2]
	\arrow[from=3-3, to=3-4]
	\arrow["\partial", from=3-3, to=4-3]
	\arrow[from=4-1, to=4-2]
	\arrow["{\nabla_1}", from=4-2, to=4-3]
	\arrow[from=4-2, to=5-2]
	\arrow[from=4-3, to=4-4]
	\arrow[from=4-3, to=5-3]
\end{tikzcd}\]

The total complex of this double complex computes $\mathbf{H}^{\bullet}_{\dR,1}(\Sh{X}_0(\pp_1)_{[a,1]}, \Fil_n\Sh{N}^{(2\kappa+k_1-2,k_2;w)})$. In particular, we have a complex as follows.

\begin{align*}
    \cdots &\to {\prod_{i<j}\Fil_n\Sh{N}^{(2\kappa+k_1-2,k_2;w)}{(\Sh{U}_{ij}})} \oplus {\prod_i\Fil_{n+1}\Sh{N}^{(2\kappa+k_1,k_2;w)}{(\Sh{U}_i})} \\
    &\to {\prod_{i<j<k}\Fil_n\Sh{N}^{(2\kappa+k_1-2,k_2;w)}{(\Sh{U}_{ijk}})} \oplus {\prod_{i<j}\Fil_{n+1}\Sh{N}^{(2\kappa+k_1,k_2;w)}{(\Sh{U}_{ij}})} \\
    &\to {\prod_{i<j<k<\ell}\Fil_n\Sh{N}^{(2\kappa+k_1-2,k_2;w)}{(\Sh{U}_{ijk\ell}})} \oplus {\prod_{i<j<k}\Fil_{n+1}\Sh{N}^{(2\kappa+k_1,k_2;w)}{(\Sh{U}_{ijk}})} \to \cdots
\end{align*}
The arrows send $(a,b) \mapsto \left(\partial(a), \nabla_1(a) - \partial(b)\right)$ where $\partial$ is the usual differential of the \v{C}ech complex. Since this complex is the mapping cone of $\nabla_1 \colon \mathscr{C}^{\bullet}(\Sh{U},\Fil_n\Sh{N}^{(2\kappa+k_1-2,k_2;w)}) \to \mathscr{C}^{\bullet}(\Sh{U},\Fil_{n+1}\Sh{N}^{(2\kappa+k_1,k_2;w)})$ the exact sequence (\ref{E03}) is clear.

\begin{prop}
    We have an isomorphism of $\Sh{O}_{\Sh{W}_r}[\lambda^{-1}]$-modules 
    \[
    H^1(\Sh{X}_0(\pp_1)_{[a,1]}, \Omega^{(2\kappa+k_1,k_2;w)}) \xrightarrow{\sim} \frac{H^1(\Sh{X}_0(\pp_1)_{[a,1]}, \Fil_{n+1}\Sh{N}^{(2\kappa+k_1,k_2;w)})}{\nabla_1\left(H^1(\Sh{X}_0(\pp_1)_{[a,1]}, \Fil_{n}\Sh{N}^{(2\kappa+k_1-2,k_2;w)})\right)}.
    \]
\end{prop}

\begin{proof}
    We first prove injectivity. We will write $Z^1(\Omega^{(2\kappa+k_1,k_2;w)})$ for $Z^1\left(\mathscr{C}^{\bullet}(\Sh{U},\Omega^{(2\kappa+k_1,k_2;w)})\right)$ and similarly for the coboundaries and other sheaves. We need to show that 
    \begin{align*}
    Z^1\left(\Fil_{n+1}\Omega^{(2\kappa+k_1,k_2;w)}\right) &\cap \left(B^1\left(\Fil_{n+1}\Sh{N}^{(2\kappa+k_1,k_2;w)}\right) + \nabla_1\left(Z^1(\Fil_n\Sh{N}^{(2\kappa+k_1-2,k_2;w)})\right)\right) \\
    &= B^1\left(\left(\Fil_{n+1}\Omega^{(2\kappa+k_1,k_2;w)}\right)\right).
    \end{align*}
    So let $(x_{ij}) \in \prod_{i<j}\Omega^{(2\kappa+k_1,k_2;w)}(\Sh{U}_{ij})$ be a 1-cocycle, such that $x_{ij} = (y_i-y_j) + \nabla_1(z_{ij})$ with $(y_i) \in \prod_i \Fil_{n+1}\Sh{N}^{(2\kappa+k_1,k_2;w)}(\Sh{U}_i)$ and $(z_{ij}) \in Z^1(\Fil_n\Sh{N}^{(2\kappa+k_1-2,k_2;w)})$. By Lemma \ref{L1030}, for each $i$ we can write 
    \[
    y_i = \psi_{\Sh{U}_i}(y_i) + \nabla_1(a_i)
    \]
    for some $a_i \in \Fil_n\Sh{N}^{(2\kappa+k_1-2,k_2;w)}(\Sh{U}_i)$. Therefore, \[x_{ij} - \left(\psi_{\Sh{U}_i}(y_i)_{|\Sh{U}_{ij}} - \psi_{\Sh{U}_j}(y_j)_{|\Sh{U}_{ij}}\right) = \nabla_1(a_i)_{|\Sh{U}_{ij}} - \nabla_1(a_j)_{|\Sh{U}_{ij}} + \nabla_1(z_{ij}).\]
    But by Griffiths' transversality, and more precisely by the description of $\nabla_1$ in local coordinates as in Proposition \ref{P1029}, $\mathrm{im}(\nabla_1) \cap \Fil_0 = 0$. Therefore, both sides of the above equation are 0, and $(x_{ij})$ is a 1-coboundary. 

    Next we prove surjectivity by constructing the inverse map. Let now $(y_{ij}) \in \prod_{i<j} \Fil_{n+1}\Sh{N}^{(2\kappa+k_1,k_2;w)}(\Sh{U}_{ij})$ be a 1-cocycle. We will construct an element in $H^1(\Sh{X}_0(\pp_1)_{[a,1]}, \Omega^{(2\kappa+k_1,k_2;w)})$ corresponding to $(y_{ij})$. By Lemma \ref{L1030}, write $y_{ij} = \psi_{\Sh{U}_{ij}}(y_{ij}) + \nabla_1(z_{ij})$ for some $z_{ij} \in \Fil_n\Sh{N}^{(2\kappa+k_1-2,k_2;w)}(\Sh{U}_{ij})$. By the canonical construction of the maps $\psi_{\Sh{U}}$ we have $\partial(\psi_{\Sh{U}_{ij}}(y_{ij})) = (\psi_{\Sh{U}_{ijk}}(\partial(y_{ij}))) = 0$. Therefore $(\psi_{\Sh{U}_{ij}}(y_{ij}))$ defines an 1-cocycle with coefficients $\Omega^{(2\kappa+k_1,k_2;w)}$. It is then an easy check to show that this defines a well-defined map 
    \[
    \frac{H^1(\Sh{X}_0(\pp_1)_{[a,1]}, \Fil_{n+1}\Sh{N}^{(2\kappa+k_1,k_2;w)})}{\nabla_1\left(H^1(\Sh{X}_0(\pp_1)_{[a,1]}, \Fil_{n}\Sh{N}^{(2\kappa+k_1-2,k_2;w)})\right)} \to H^1(\Sh{X}_0(\pp_1)_{[a,1]}, \Omega^{(2\kappa+k_1,k_2;w)})
    \]
    which is clearly the inverse of the natural map in the statement of the Proposition.
\end{proof}

\begin{defn}
    Define the overconvergent projection to be the map
    \begin{align*}
    e^{\dagger}\colon H^1(\Sh{X}_0(\pp_1)_{[a,1]}, &\Fil_{n+1}\Sh{N}^{(2\kappa+k_1,k_2;w)}) \\
    &\to \frac{H^1(\Sh{X}_0(\pp_1)_{[a,1]}, \Fil_{n+1}\Sh{N}^{(2\kappa+k_1,k_2;w)})}{\nabla_1\left(H^1(\Sh{X}_0(\pp_1)_{[a,1]}, \Fil_{n}\Sh{N}^{(2\kappa+k_1-2,k_2;w)})\right)} \xrightarrow{\sim} H^1(\Sh{X}_0(\pp_1)_{[a,1]}, \Omega^{(2\kappa+k_1,k_2;w)}).
    \end{align*}
\end{defn}

\subsection{Projection to eigenspaces}
The cohomology complex $R\Gamma(\Sh{X}_0(\pp_1)_{[a,1]}, \Omega^{(2\kappa+k_1,k_2;w)})^{\leq h}$ is also equipped with Hecke operators prime to $\pp_1$ and the level $K^{(p)}$. Let $\mathbb{T}$ denote the abstract algebra over $K = R[1/p]$ generated by formal symbols corresponding to these operators and the $U_{\pp_1}$ operator. Then $\mathbb{T}$ acts on $R\Gamma(\Sh{X}_0(\pp_1)_{[a,1]}, \Omega^{(2\kappa+k_1,k_2;w)})^{\leq h}$.

\begin{prop}
    The derived cohomology $R\Gamma(\Sh{X}_0(\pp_1)_{[a,1]}, \Omega^{(2\kappa+k_1,2-k_2;w)})^{\leq h}$ can be represented by a perfect complex $M^{\bullet}(2\kappa+k_1, 2-k_2;w)$ of $\Sh{O}_{\Sh{W}_r}$-modules.
\end{prop}

\begin{proof}
    This follows from \cite[Theorem 6.4.3]{boxer2021highercolemantheory}.
\end{proof}

\begin{cor}
    For any integer weight specialisation $\kappa \mapsto a$ we have 
    \[
    M^{\bullet}(2\kappa + k_1, 2-k_2;w) \otimes_{\Sh{O}(\Sh{W}_r)}^{\kappa = a} K \xrightarrow{\sim} M^{\bullet}(2a+k_1,2-k_2;w) = R\Gamma(\Sh{X}_0(\pp_1)_{[a,1]}, \underline{\omega}^{(2a+k_1,2-k_2;w)})^{\leq h}.
    \]
    Moreover, for $2a+k_1 -1 > h$, we have a quasi-isomorphism
    \[
    R\Gamma(\Sh{X}_0(\pp_1), \underline{\omega}^{(2a+k_1,2-k_2;w)})^{\leq h} \simeq R\Gamma(\Sh{X}_0(\pp_1)_{[a,1]}, \underline{\omega}^{(2a+k_1,2-k_2;w)})^{\leq h}.
    \]
\end{cor}

\begin{proof}
    The first claim is clear since $M^{\bullet}(2\kappa+k_1,2-k_2;w)$ is a perfect complex. Let us briefly explain the second claim which is a classicality result. One can consider the compactly supported cohomology $R\Gamma_{\Sh{X}_0(\pp_1)_{[0,a[}}(\Sh{X}_0(\pp_1), \underline{\omega}^{(2a+k_1,2-k_2;w)})$ which sits in the distinguished triangle
    \[
    R\Gamma_{\Sh{X}_0(\pp_1)_{[0,a[}}(\Sh{X}_0(\pp_1), \underline{\omega}^{(2a+k_1,2-k_2;w)}) \to R\Gamma(\Sh{X}_0(\pp_1), \underline{\omega}^{(2a+k_1,2-k_2;w)}) \to R\Gamma(\Sh{X}_0(\pp_1)_{[a,1]}, \underline{\omega}^{(2a+k_1,2-k_2;w)}) \xrightarrow{+1}
    \]
    Then the same proof as in \cite[Lemma 5.11]{BoxerPilloni_ModularCurve} shows that $U_{\pp_1}$ acts with slope $\geq 0$ on $R\Gamma(\Sh{X}_0(\pp_1)_{[a,1]}, \underline{\omega}^{(2a+k_1,2-k_2;w)})$, and with slope $\geq 2a+k_1-1$ on $R\Gamma_{\Sh{X}_0(\pp_1)_{[0,a[}}(\Sh{X}_0(\pp_1), \underline{\omega}^{(2a+k_1,2-k_2;w)})$. The claim follows by applying $e^{\leq h}$.
\end{proof}

Let $\Pi$ be a $\pp_1$-stabilisation of a classical cuspidal Hecke eigenform of weight $(k_1,k_2;w)$ with $k_1, k_2 \geq 2$ and $U_{\pp_1}$-slope $\leq h < k_1-1$, and assume it is defined over $K$. For a suitable choice of the level $K^{(p)}$ away from $p$ (of the form $\{ \gamma : \gamma = \left(\begin{smallmatrix} \star & \star \\ 0 & 1 \end{smallmatrix}\right) \bmod \mathfrak{N}\}$ where $\mathfrak{N}$ is the conductor of $\Pi$), the automorphic representation corresponding to $\Pi$ contributes to $H^1(\Sh{X}_0(\pp_1), \underline{\omega}^{(1\pm(k_1-1),1\pm (k_2-1);w)})$, each with multiplicity 1. We will show that the corresponding $H^1(\Sh{X}_0(\pp_1), \underline{\omega}^{(k_1,2-k_2;w)})$ class can be interpolated in a Coleman family of slope $\leq h$ in $H^1$ of the overconvergent cohomology, in a small enough neighbourhood of $k_1$ inside $\Sh{W}_r$.

\begin{lemma}
    $H^1(\Sh{X}_0(\pp_1)_{[a,1]}, \Omega^{(2\kappa+k_1,2-k_2;w)})^{\leq h}$ is locally free in a neighbourhood of $\kappa = 0$.
\end{lemma}

\begin{proof}
    By the local criterion of flatness, we need to show that \[\Tor_1^{\Sh{O}_{\Sh{W}_r}}(H^1(\Sh{X}_0(\pp_1)_{[a,1]}, \Omega^{(2\kappa+k_1,2-k_2;w)})^{\leq h}, K) = 0\]
    where $\Sh{O}_{\Sh{W}_r} \to K$ defines the point $\kappa = 0$. By \cite[Theorem 6.5.2]{boxer2021highercolemantheory} this $\Tor$ group is a graded piece of $H^0(\Sh{X}_0(\pp_1)_{[a,1]}, \underline{\omega}^{(k_1,2-k_2;w)})^{\leq h} \simeq H^0(\Sh{X}_0(\pp_1), \underline{\omega}^{(k_1,2-k_2;w)})^{\leq h}$ which is $0$ since $k_2 \geq 2$ \cite[Ch 2, Prop. 3.6]{GorenHilbertbook}.
\end{proof}

Let $\Sh{U}$ be a sufficiently small neighbourhood of $\kappa = 0$ over which $H^1(\Sh{X}_0(\pp_1)_{[a,1]}, \underline{\omega}^{(2\kappa+k_1,2-k_2;w)})^{\leq h}$ is finite, locally free, and call this module $H^1(2\kappa_{\Sh{U}}+k_1, 2-k_2;w)^{\leq h}$ which for brevity of notation we'll write as $H^1(2\kappa_{\Sh{U}})^{\leq h}$. Similarly, write $H^2(2\kappa_{\Sh{U}}+k_1, 2-k_2;w)^{\leq h}$ and its shorthand $H^2(2\kappa_{\Sh{U}})^{\leq h}$. The Tor-spectral sequence \cite[Theorem 6.5.2]{boxer2021highercolemantheory} referred to above gives a short exact sequence
\begin{equation}\label{E04}
0 \to H^1(2\kappa_{\Sh{U}})^{\leq h}{\otimes}_{\Sh{O}(\Sh{U})}^{\kappa=0} K \to H^1(k_1,2-k_2;w)^{\leq h} \to \Tor_1^{\Sh{O}(\Sh{U})}(H^2(2\kappa_{\Sh{U}})^{\leq h}, K) \to 0.
\end{equation}
Let $I_{\Pi}$ be the kernel of $\Sh{O}(\Sh{U}) \otimes \mathbb{T} \xrightarrow{(\kappa=0) \otimes \Pi} K$. 

\begin{prop}\label{P3018}
    Localising sequence (\ref{E04}) at $I_{\Pi}$ induces an isomorphism 
    \[
    H^1(2\kappa_{\Sh{U}})^{\leq h}_{I_{\Pi}}{\otimes}_{\Sh{O}(\Sh{U})}^{\kappa=0} K \xrightarrow{\sim} H^1(k_1,2-k_2;w)_{I_{\Pi}}^{\leq h}.
    \]
    In particular, $H^1(2\kappa_{\Sh{U}})^{\leq h}_{I_{\Pi}}$ is free of rank 1 over the localisation of $\Sh{O}(\Sh{U})$ at $\kappa = 0$, and by possibly shrinking $\Sh{U}$, there is a unique homomorphism $\underline{\Pi} \colon \mathbb{T} \to \Sh{O}(\Sh{U})$ which specialises to $\Pi$ at $\kappa = 0$.
\end{prop}

\begin{proof}
Consider the projective resolution $0 \to {\Sh{O}}(\Sh{U})_{(t)} \xrightarrow{\times t} {\Sh{O}}(\Sh{U})_{(t)} \to K$ of $K$, where $t$ is a local coordinate around $0$ and ${\Sh{O}}(\Sh{U})_{(t)}$ is the localisation of $\Sh{O}(\Sh{U})$ at $\kappa = 0$. Since $\Pi$ doesn't contribute to $H^2(k_1,2-k_2;w)$, and hence to $H^2(2\kappa_{\Sh{U}})^{\leq h}{\otimes}_{\Sh{O}(\Sh{U})}^{\kappa=0} K$, we have an exact sequence
\[
H^2(2\kappa_{\Sh{U}})^{\leq h}_{I_{\Pi}}{\otimes}{\Sh{O}}(\Sh{U})_{(t)} \xrightarrow{\times t} H^2(2\kappa_{\Sh{U}})^{\leq h}_{I_{\Pi}}{\otimes}{\Sh{O}}(\Sh{U})_{(t)} \to \left(H^2(2\kappa_{\Sh{U}})^{\leq h}{\otimes}_{\Sh{O}(\Sh{U})}^{\kappa=0} K\right)_{I_{\Pi}}
\]
where the rightmost term is $0$. This implies that the finite $\Sh{O}(\Sh{U})_{(t)}$-module $H^2(2\kappa_{\Sh{U}}+k_1, 2-k_2;w)^{\leq h}_{I_{\Pi}} = 0$, and hence the same is true of its $\Tor_1$ group. The rest of the claims immediately follow.
\end{proof}

\section{Pushforward from the modular curve}
In this section we will consider Shimura varieties for both the groups $H$ and $G$, and hence will use subscript notation $(\cdot)_{H}, (\cdot)_G$ to distinguish them; we take the prime-to-$p$ level for $H$ to be $K_H^{(p)} \coloneqq H \cap K^{(p)}$. We note that there is a finite map 
\[
\iota \colon X_{H,0}(p) \xhookrightarrow{} X_{G,0}(\pp_1)
\]
of $\Z_p$-schemes. We change base to $K$, and consider the associated rigid analytic spaces. The embedding $\iota$ has a moduli interpretation given by sending an elliptic curve $E$ with cyclic $p$-subgroup $C \subset E[p]$ (and other prime to $p$ level data) to the abelian surface $E \otimes \Sh{O}_F$, together with $\operatorname{Iw}_{\pp_1}$ level given by $C$. Note that since the group embedding $H \to G$ factors through $G^*$, the embedding of the modular curve $\Sh{X}_{H,0}(p)$ factors through $\Sh{M}_{0}(\pp_1)^{\mathfrak{d}^{-1}}$, where the $\mathfrak{d}^{-1}$-polarisation on $E \otimes \Sh{O}_F$ is induced by the principal polarisation on $E$. The $p$-divisible group of $E \otimes \Sh{O}_F$ is given by two copies of $E[p^{\infty}]$, and hence $\deg(\iota(x)) = \deg(x)$ for any $x \in \Sh{X}_{H,0}(p)$.

Let $\Sh{E} \to \Sh{X}_{H,0}(p)$ be the universal generalised elliptic curve. We write $\underline{\omega}_{\Sh{E}}$ for the modular sheaf of weight $1$, and $\Sh{H}_{\Sh{E}}$ for the canonical extension of the relative de Rham sheaf of the universal elliptic curve. As in the case of $G$, we have a Hodge filtration
\[
0 \to \underline{\omega}_{\Sh{E}} \to \Sh{H}_{\Sh{E}} \to \underline{\omega}_{\Sh{E}^{\vee}}^{\vee} \to 0.
\]
For any integer $k$, we denote the modular sheaf of weight $k$ by $\underline{\omega}_H^k := \underline{\omega}_{\Sh{E}}^k$. As in the theory developed in $\S\ref{S2}$ there exists integral models $\omega_{\Sh{E}}^{\sharp}$ and $\Sh{H}_{\Sh{E}}^{\sharp}$ of $\underline{\omega}_{\Sh{E}}$ and $\Sh{H}_{\Sh{E}}$ respectively, that sit in a short exact sequence induced by the Hodge filtration
\[
0 \to \omega_{\Sh{E}}^{\sharp} \to \Sh{H}_{\Sh{E}}^{\sharp} \to \underline{\Ha}^{\frac{p}{p-1}}\underline{\omega}_{\Sh{E}}^{\vee} \to 0.
\]
Moreover $\omega_{\Sh{E}}^{\sharp}$ is equipped with a marked section $s_H$ modulo $p^{n-v\frac{p^n}{p-1}}$. Hence there exists a nearly overconvergent sheaf $\Sh{N}_H^{2\kappa}$ of weight $2\kappa$ over the base change of $\Sh{X}_{H,0}(\pp_1)_{[a,1]}$ to $\Sh{W}_r$. 

Let $D_H,D_G$ be the cuspidal divisors of $X_{H,0}(p), X_{G,0}(\pp_1)$ respectively. The dualising sheaves of $\Sh{X}_{H,0}(p)$ and $\Sh{X}_{G,0}(\pp_1)$ are $\Omega_{\Sh{X}_{H,0}(p)/K} \simeq \underline{\omega}_H^2(-D_H) \otimes (\wedge^2\Sh{H}_{\Sh{E}})^{-1}$ and $\Omega^2_{\Sh{X}_{G,0}(\pp_1)/K} \simeq \underline{\omega}_G^{(2,2;0)}(-D_G)$.

We have a pushforward map 
\[
H^0\left(\Sh{X}_{H,0}(p)_{[a,1]}, \iota^*(\Sh{N}_G^{(2\kappa+k_1-2,-k_2;w)}(D_G)) \otimes \underline{\omega}_H^2(-D_H)\right) \to H^1(\Sh{X}_{G,0}(\pp_1)_{[a,1]}, \Sh{N}_G^{(2\kappa+k_1,2-k_2;w)})
\]

\begin{lemma}
    There is a map 
    \[
    H^0\left(\Sh{X}_{H,0}(p)_{[a,1]}, \Sh{N}_H^{2\kappa+k_1-k_2}\right) \otimes |\det(\cdot)|_{\A{\Q}{(p)}}^{\frac{2\kappa+k_1-k_2}{2}-w} \to H^0\left(\Sh{X}_{H,0}(p)_{[a,1]}, \iota^*(\Sh{N}_G^{(2\kappa+k_1-2,-k_2;w)}(D_G)) \otimes \underline{\omega}_H^2(-D_H)\right)
    \]
\end{lemma}

\begin{rem}
    The $|\det|$ factor is to ensure that the central characters match up.
\end{rem}

\begin{proof}
    It is enough to construct a map of sheaves 
    \[
    \Sh{N}_H^{2\kappa+k_1-2-k_2} \to \iota^*\Sh{N}_G^{(2\kappa+k_1-2,-k_2;w)}.
    \]
    In fact since $\Sh{N}_G^{(2\kappa+k_1-2,-k_2;w)} \simeq \Sh{N}_G^{(2\kappa,0;0)}\hat{\otimes}\underline{\omega}_G^{(k_1-2,-k_2;w)}$, and $\iota^*\underline{\omega}_G^{(k_1-2,-k_2;w)} \simeq \underline{\omega}_H^{k_1-2-k_2}$ canonically, it is enough to construct the map $\Sh{N}_H^{2\kappa} \to \iota^*\Sh{N}_G^{(2\kappa,0;0)}$. For this we note that working integrally over the formal models of the Shimura varieties with spherical level at $p$, and corresponding Igusa varieties, \[\iota^*\Sh{H}_{\Sh{A},1} \oplus \iota^*\Sh{H}_{\Sh{A},2} = \iota^*\Sh{H}_{\Sh{A}} \simeq \Sh{H}_{\Sh{E}}\otimes \Sh{O}_F = \Sh{H}_{\Sh{E}}\oplus \Sh{H}_{\Sh{E}}.\] Moreover, since $(\Sh{E}\otimes \Sh{O}_F)[p^{\infty}] = \Sh{E}[p^{\infty}] \oplus \Sh{E}[p^{\infty}]$, we have a canonical identification $\Sh{H}^{\sharp}_{\Sh{E}} \simeq \iota^*\Sh{H}_{\Sh{A},1}^{\sharp}$ that sends the marked section to the marked section and induces isomorphism on the graded pieces of the respective Hodge filtration. Hence we get a canonical isomorphism $\Sh{N}_H^{2\kappa} \simeq \iota^*\Sh{N}_G^{(2\kappa,0;0)}$.  
\end{proof}

\begin{cor}
    We have a pushforward map
    \[
    \iota_* \colon H^0\left(\Sh{X}_{H,0}(p)_{[a,1]}, \Sh{N}_H^{2\kappa+k_1-k_2}\right) \otimes |\det(\cdot)|_{\A{\Q}{(p)}}^{\frac{2\kappa+k_1-k_2}{2}-w} \xrightarrow{} H^1(\Sh{X}_{G,0}(\pp_1)_{[a,1]}, \Sh{N}_G^{(2\kappa+k_1,2-k_2;w)}).
    \]
\end{cor}

\section{Adelic Eisenstein series and \texorpdfstring{$p$}{p}-adic Eisenstein measure}\label{S05}
In this section we recall generalities about adelic Eisenstein series, and the two variable Katz $p$-adic Eisenstein measure following \cite{LPSZ} and \cite{grossi2024padicasailfunctionsquadratic}.

For a Schwartz function $\Phi \colon \A{}{2} \to \C$, and $\chi$ a Dirichlet character with corresponding adelic Hecke character $\hat{\chi}$, define a global Siegel section and associated Eisenstein series by 
\[
    f^{\Phi}(g;\chi, s) := |\det{g}|_{\A{}{}}^s \int_{\A{}{\times}} \Phi((0,a)g)\hat{\chi}(a)|a|_{\A{}{}}^{2s}\mathrm{d}^{\times}a, 
\]
\[
    E^{\Phi}(g;\chi,s) := \sum_{\gamma \in B(\Q)\backslash\GL_2(\Q)} f^{\Phi}(\gamma g;\chi,s).
\]
The sum converges absolutely and uniformly on compact subsets for $\mathrm{Re}(s) > 1$, has a meromorphic continuation in $s$, and defines a function on the quotient $\GL_2(\Q)\backslash \GL_2(\A{}{})$ with central character $\hat{\chi}^{-1}$.

Choosing the Schwartz function at $\infty$ to be 
\[
\Phi_{\infty}^{(k)}(x,y) := 2^{1-k}(x+iy)^k e^{-\pi(x^2+y^2)} 
\]
for any $k \in \Z_{\geq 1}$, and letting $\Phi = \Phi_{\infty}^{(k)}\Phi_{\mathrm{f}}$ for any choice of $\Phi_{\mathrm{f}} \in \Sh{S}(\A{\mathrm{f}}{2}, \C)$, the function $E^{k,\Phi_{\mathrm{f}}}(-; \chi, s) := E^{\Phi_{\infty}^{(k)}\Phi_{\mathrm{f}}}(-;\chi,s)$ is a nearly holomorphic modular form of weight $k$ for any $s$ with $1-k/2\leq s \leq k/2$ with $s \equiv k/2\ \mathrm{mod}\ \Z$, which is moreover holomorphic for the two end values $s = 1-k/2$ and $k/2$. Moreover, if $\Phi_{\mathrm{f}}$ and $\chi$ takes value in a number field $L$, then this form is defined over $L$ as a coherent cohomology class. 

We now fix a Schwartz function $\Phi^{(p)}$ away from $p$, and $\chi^{(p)}$ a Dirichlet character of conductor coprime to $p$ such that $\left(\begin{smallmatrix}
    a & 0 \\ 0 & a
\end{smallmatrix}\right)\cdot \Phi^{(p)} = \hat{\chi}^{(p)}(a)^{-1}\Phi^{(p)}$ for $a \in (\hat{\Z}^{(p)})^{\times}$.

Let us now consider a pair $\mu, \nu$ of Dirichlet characters of $p$-power conductor. We will attach to such a pair an element $\Phi_{p,\mu,\nu} \in \Sh{S}(\Q_p^2, \C)$. Let $\psi \colon \A{}{}/\Q \to \C$ be the unique additive character whose restriction to the infinite place is given by $\psi_{\infty}(x_{\infty}) = e^{-2\pi ix_{\infty}}$. Then letting $\psi_p$ be its restriction to $\Q_p$, we have the Fourier transform of any element $\phi \in \Sh{S}(\Q_p,\C)$ given by 
\[
\hat{\phi}(x) = \int_{\Q_p} \phi(y)\psi_p(xy)\mathrm{d}y.
\]
For any character $\xi \colon \Q_p^{\times} \to \C^{\times}$, let $\phi_{\xi}$ be the Schwartz function defined by $\phi_{\xi}(x) = \mathbf{1}_{\Z_p^{\times}}(x)\xi(x)$. We then let 
\[
\Phi_{p,\mu,\nu} := \phi_{\mu}(x)\hat{\phi}_{\nu}(y),
\]
regarded as a function taking values in $\bar{\Q}_p$ under our fixed isomorphism $\bar{\Q}_p \simeq \C$.

Recall $\Lambda = R\llbracket \Z_p^{\times} \rrbracket$, where $R$ is the ring of integers of a finite extension $K$ of $\Q_p$. Let $\Lambda_K = \Lambda[1/p]$, and $\Lambda'_K = \Lambda_K \hat{\otimes} \Lambda_K$. Let $\mathfrak{X}_H^m$ be the p-adic completion of the multiplicative locus of $X_{H,0}(p)$, viewed as a formal scheme over $\Lambda$, and let $\Omega^{\kappa}$ be the sheaf of Katz $p$-adic modular forms of universal weight $\kappa$ over $\mathfrak{X}_H^m$.
\begin{theorem}[cf. Theorem 7.2.1 \cite{grossi2024padicasailfunctionsquadratic}]
    Let $(\kappa_1, \kappa_2)$ be the universal character of $\Lambda'_K$. Then for each $\Phi^{(p)}$ taking values in an algebraic number field $L \subset \C$ which is contained inside $K$ by the fixed isomorphism $\bar{\Q}_p \simeq \C$, there exists a two variable measure
    \[
    \Sh{E}^{\Phi^{(p)}}(\kappa_1,\kappa_2;\chi^{(p)}) \in H^0(\mathfrak{X}_H^m, \Omega^{\kappa}) \underset{{\Lambda, \kappa_1+\kappa_2+1}}{\hat{\otimes}} \Lambda'_K
    \]
    such that its specialisation at $(a+\mu, b+\nu)$ for any $a,b \geq 0$ is the $p$-adic modular form associated to the $p$-depleted algebraic nearly holomorphic modular form of weight $a+b+1$
    \[
    g \mapsto \hat{\nu}(\det g)^{-1}\cdot E^{(a+b+1, \Phi^{(p)}\Phi_{p,\mu,\nu})}\left(g;\chi^{(p)}\mu\nu^{-1},\frac{b-a+1}{2}\right).
    \]
    For any $t \in \Z_{\geq 2}$ even, we can construct a one parameter family of Eisenstein series $\Sh{E}^{\Phi^{(p)}}_t(\kappa;\chi^{(p)}) \in H^0(\mathfrak{X}_{H}^m, {\Omega}_H^{t}) \hat{\otimes} \Lambda_K$ by setting \[\Sh{E}_t^{\Phi^{(p)}}(\kappa;\chi^{(p)}) = \Sh{E}^{\Phi^{(p)}}\left(\frac{t}{2}-\kappa, \frac{t}{2} + \kappa - 1;\chi^{(p)}\right)\]
    such that its specialisation at $s+\nu$ for any $s \in \Z$ with $1-t/2\leq s \leq t/2$ is the $p$-adic modular form associated to the $p$-depleted algebraic nearly holomorphic modular form
    \[
    g \mapsto \hat{\nu}(\det g)^{-1}\cdot E^{(t, \Phi^{(p)}\Phi_{p,\nu^{-1},\nu})}\left(g;\chi^{(p)}\nu^{-2},s\right).
    \]
\end{theorem}

\begin{rem}\label{R405}
    Some remarks are in order.
    \begin{enumerate}
        \item The specialisation of $\Sh{E}^{\Phi^{(p)}}(\kappa_1,\kappa_2;\chi^{(p)})$ to locally algebraic weights of the form $(a+\mu,0)$ or $(0,b+\nu)$ is not only nearly holomorphic, but is honestly holomorphic. In particular restriction of $\Sh{E}^{\Phi^{(p)}}(\kappa_1,\kappa_2;\chi^{(p)})$ to either the line $\kappa_1 = 0$ or $\kappa_2 = 0$ is an overconvergent family of modular forms of weight $\kappa_2+1$ (resp. $\kappa_1+1$) \cite[Lemma 3.2]{LoefflerRankin}.
        \item Let $\Sh{E}^{\Phi^{(p)}}(0,\kappa_2-\kappa_1;\chi^{(p)}) \in H^0(\mathfrak{X}_H^m, \Omega^{\kappa}) \underset{{\Lambda, \kappa_2-\kappa_1+1}}{\hat{\otimes}} \Lambda'_K$ be the overconvergent specialisation as above given by $\Sp \Lambda'_K \to \Sp \Lambda'_K$, $(\kappa_1,\kappa_2) \mapsto (0,\kappa_2-\kappa_1)$. The $q$-expansion of $\Sh{E}^{\Phi^{(p)}}(\kappa_1,\kappa_2;\chi^{(p)})$ at $\infty$ is given by (cf. \cite[Theorem 7.6]{LPSZ}) 
        \[
        \sum_{\substack{u,v \in (\Z_{(p)}^{\times})^2\\ uv>0}} \mathrm{sgn}(u)u^{\kappa_1}v^{\kappa_2}(\hat{\Phi}_{\mathrm{f}}^{(p)})(u,v)q^{uv}
        \]
        where $\hat{\Phi}_{\mathrm{f}}^{(p)}(u,v) := \int_{\A{\mathrm{f}}{(p)}}\Phi_{\mathrm{f}}^{(p)}(u,w)\psi(uw)\mathrm{d}w$. In particular, letting $\theta = q\frac{\mathrm{d}}{\mathrm{d}q}$, we have 
        \[
        \Sh{E}^{\Phi^{(p)}}(\kappa_1,\kappa_2;\chi^{(p)}) = \theta^{\kappa_1}\Sh{E}^{\Phi^{(p)}}(0,\kappa_2-\kappa_1;\chi^{(p)}).
        \]
    \end{enumerate}
\end{rem}

We now explain how to lift $\Sh{E}^{\Phi^{(p)}}(\kappa_1,\kappa_2;\chi^{(p)})$ to a nearly overconvergent family using the results of \cite{andreatta2021triple}. For this, let us fix an open affinoid $\Sh{U}$ inside the weight space. We will later be defining our $p$-adic $L$-function by pairing the pushforward of a nearly overconvergent family of Eisenstein series against a Coleman family $\underline{\Pi}$ of the form afforded by Proposition \ref{P3018}, hence the reader can think of $\Sh{U}$ as in that proposition. 

We recall that the $\theta$ operator can be interpreted in terms of the Gauss–Manin connection as follows.
\[
\theta \colon \underline{\omega}^k \to \Sym^k\Sh{H}_{\Sh{E}} \xrightarrow{\nabla} \Sym^k\Sh{H}_{\Sh{E}} \otimes \Omega^1_{\mathfrak{X}_H^{\mathrm{ord}}}(\log D_H) \simeq \Sym^{k+2}\Sh{H}_{\Sh{E}} \xrightarrow{\mathrm{{Spl}_{ur}}} \underline{\omega}^{k+2}.
\]
Here $\mathrm{{Spl}_{ur}}$ is the unit-root splitting, which only exists over the ordinary locus.

\begin{theorem}\label{T406}
Let $\Sh{V}$ be an open affinoid in the weight space.
    \begin{enumerate}
        \item For large enough $a$, and letting $(\kappa_1,\kappa_2)$ be the universal character of $\Sh{U\times V}$, we have \[\Sh{E}^{\Phi^{(p)}}(0,\kappa_2-\kappa_1;\chi^{(p)}) \in H^0(\Sh{X}_{H,0}(p)_{[a,1]}, \Omega^{\kappa}) \underset{\Sh{O}(\Sh{W}_r),\kappa_2-\kappa_1+1}{\hat{\otimes}}\Sh{O}(\Sh{U} \times \Sh{V}).\]
        \item There exists a nearly overconvergent family 
        \[
        \nabla^{\kappa_1}\Sh{E}^{\Phi^{(p)}}(0,\kappa_2-\kappa_1;\chi^{(p)}) \in H^0(\Sh{X}_{H,0}(p)_{[a',1]}, \Sh{N}^{\kappa})\underset{\Sh{O}(\Sh{W}_r),\kappa_1+\kappa_2+1}{\hat{\otimes}} \Sh{O}(\Sh{U\times V})
        \]
        for some $a' \geq a$, whose image under $\mathrm{Spl_{ur}}$ after restricting to the multiplicative locus $\mathfrak{X}_H^m$ is given by $\Sh{E}^{\Phi^{(p)}}(\kappa_1,\kappa_2;\chi^{(p)})$. Moreover, this is unique up to shrinking the radius of overconvergence.
    \end{enumerate}
\end{theorem}

\begin{proof}
    Part $(1)$ is restating Remark \ref{R405} (1). Part (2) is the main result of \cite[Theorem 4.3]{andreatta2021triple}. We note that the $p$-adic Eisenstein series $\Sh{E}^{\Phi^{(p)}}(0,\kappa_2-\kappa_1;\chi^{(p)})$ is $p$-depleted, and thus lies in the kernel of $U_p$. Therefore we can use the results of loc. cit. The uniqueness is a corollary of the injectivity of $\mathrm{Spl_{ur}}$.
\end{proof}

\begin{rem}
    Strictly speaking Theorem 4.3 of \cite{andreatta2021triple} only holds for $\kappa_1,\kappa_2$ satisfying certain strong analyticity conditions (cf. Assumption 4.1 in loc. cit.). One can ensure these conditions by shrinking $\Sh{U}$ and $\Sh{V}$ if necessary. Or else, one can construct certain bigger sheaves $\Sh{N}^{\kappa,\mathrm{new}} \supset \Sh{N}^{\kappa}$ following the discussion in Remark \ref{R2010}. As \cite{graham2023padic} shows, using these bigger sheaves allows one to relax the analyticity conditions on $\kappa_2-\kappa_1$ and $\kappa_1$; we leave the details to the interested reader.
\end{rem}

\section{\texorpdfstring{$p$}{p}-adic Asai \texorpdfstring{$L$}{L}-function}\label{S5}

 \subsection{Whittaker models and periods}
 
  Let $\Pi$ be a cuspidal automorphic representation of weight $(k_1,k_2)$ defined over a number field $L$ which we assume is contained in $K$ as before. We assume $k_1>k_2\geq 1$, $k_1 = k_2 \, \mathrm{mod}\, 2$, and $\Pi$ is of level $\mathfrak{N}$ prime to $p$.

  We recall the following facts from \cite{grossi2024padicasailfunctionsquadratic}.

  \begin{enumerate}
    \item The $K_1(\mathfrak{N})$-invariants of the Whittaker model (with respect to the additive character $\psi_F(y) = \psi(\Tr_{F/\Q}(y/\sqrt{D}))$) of $\Pi_{\mathrm{f}}$, denoted $\Sh{W}(\Pi_{\mathrm{f}})$, is $1$-dimensional and has a unique basis $W_{\mathrm{f}}^{\text{new}}$ with $W_{\mathrm{f}}^{\text{new}}(1) = 1$.
    \item Letting $X_G$ be the compactified Shimura variety as in $\S\ref{S2}$ for level $K_1(\mathfrak{N})$, the $L$-vector space
    \[
    H^1\left(X_{G,L}, \underline{\omega}_G^{(2-k_1,k_2;w)}\right)[\Pi_{\mathrm{f}}]
    \]
    is $1$-dimensional, and a choice of a basis $\nu_{\Pi}$ determines an isomorphism  of $\GL_2(\A{F,\mathrm{f}}{})$-representations
    \[
    \Sh{W}(\Pi_{\mathrm{f}})^{K_1(\mathfrak{N})}\otimes |\det(\cdot)|^{-w/2} \xrightarrow{\sim} H^1\left(X_{G,\C}, \underline{\omega}_G^{(2-k_1,k_2;w)}\right)[\Pi_{\mathrm{f}}]
    \]
    given by sending $W^{\text{new}}_{\mathrm{f}} \mapsto \nu_{\Pi}$.
    \item Define the normalised Whittaker function at $\infty$ as the function on $\GL_2(F \otimes \R)$ which is supported on the identity component and satisfies
    \[
    W_{\infty}^{\mathrm{ah},1}\left(\begin{pmatrix}
        y & x \\ 0 & 1
    \end{pmatrix}\begin{pmatrix}
        t\cos{\theta} & t\sin{\theta} \\
        -t\sin{\theta} & t\cos{\theta}
    \end{pmatrix} \right) = \mathrm{sgn}(t)^{\underline{k}}\cdot y^{\underline{k}/2}\cdot e^{i(-k_1\theta_1 + k_2\theta_2)}\cdot e^{2\pi i((-x_1+iy_1)+(x_2+iy_2))/\sqrt{D}}
    \]
    for $x,y,t,\theta \in F \otimes \R$, $y \gg 0$, and $D$ the discriminant of $F/\Q$. This is a basis of the minimal $K$-type subspace in the Whittaker model of the $\sigma_1$-antiholomorphic, $\sigma_2$-holomorphic part of $\Pi_{\infty}$. 
    \item Let $\phi_{\text{new}}^{\mathrm{ah},1}$ be the unique function in $\Pi$ whose Whittaker function factors as $W^{\mathrm{ah},1}_{\infty}W^{\text{new}}_{\mathrm{f}}$. Define the period $\Omega_{\infty}(\Pi) \in \C^{\times}$ to be the scalar such that $\phi^{\mathrm{ah},1}_{\text{new}} = \Omega(\Pi)\nu_{\Pi}$.
  \end{enumerate}

  We choose a $\mathfrak{p}_1$-stabilisation $\nu_{\Pi,\alpha_1}$ of $\nu_{\Pi}$ in $H^1(X_{G,0}(\mathfrak{p}_1)_{K}, \underline{\omega}_G^{(2-k_1,k_2;w)})$ with $U_{\mathfrak{p}_1}^t$ eigenvalue $\alpha_1$ having slope $\leq h$. By Proposition \ref{P3018}, there exists an open affinoid neighbourhood $\Sh{U}$ of $0$ in the weight space, such that the Hecke eigensystem corresponding to the $\mathfrak{p}_1$-stabilisation of $\Pi^{\vee}$ with $U_{\mathfrak{p}_1}$ eigenvalue $\alpha_1$ can be interpolated over $\Sh{U}$ in a Coleman family of slope $\leq h$, which we denote by $\underline{\Pi}^{\vee}$. Moreover, the specialisation of $\underline{\Pi}^{\vee}$ at an integer weight $\kappa_{\Sh{U}} = a$, denoted $\Pi[a]^{\vee}$ is the Hecke eigensystem corresponding to a $\mathfrak{p}_1$-stabilisation of an automorphic representation if $a$ is large enough (e.g. $k_1+2a - 1 > 2h$ to avoid the $\GL_2(F_{\mathfrak{p}_1})$-representation being twists of Steinberg). Henceforth, by shrinking $\Sh{U}$ if necessary, we assume that all specialisations with $a\geq 0$ are non-Steinberg.

Fix an arbitrary isomorphism 
\[
\nu_{\underline{\Pi}} \colon H^1(2\kappa_{\Sh{U}}+k_1, 2-k_2;w)^{\leq h}[\underline{\Pi}^{\vee}] \xrightarrow{\sim} \Sh{O}(\Sh{U}).
\]
The specialisation of $\nu_{\underline{\Pi}}$ at any non-negative integer weight $a$ is a basis of the $\Pi[a]$-eigenspace in $H^1(X_{G,0}(\mathfrak{p}_1)_K, \underline{\omega}_G^{(2-k_1-2a,k_2;w)})$ which is a priori defined over $K$. Choosing an algebraic basis $\nu_{\Pi[a]}$ of the non-$\mathfrak{p}_1$-stabilised eigenspace over the coefficient field $L_{a}$ of $\Pi[a]$, we get a pair of periods as follows.
\[
\Omega_{\infty}(\Pi[a]) \in \C^{\times}, \quad \phi^{\mathrm{ah},1}_{\text{new},\Pi[a]} = \Omega_{\infty}(\Pi[a])\cdot \nu_{\Pi[a]}
\]
\[
\Omega_p(\Pi[a]) \in K^{\times},\quad (\nu_{\underline{\Pi}})_{|a} = \Omega_p(\Pi[a])\cdot \nu_{\Pi[a],\mathfrak{p}_1}
\]
Here $\nu_{\Pi[a],\mathfrak{p}_1}$ is the appropriate $\mathfrak{p}_1$-stabilisation of $\nu_{\Pi[a]}$. We note that $\Omega_p(\Pi[a])/\Omega_{\infty}(\Pi[a]) \in \C\otimes_{L_{a}} K$ is independent of the choice of $\nu_{\Pi[a]}$ for all $a\geq 0$.

  \subsection{The $p$-adic $L$-function} 
  
   We now choose a Schwartz function $\Phi^{(p)} = \prod_{\ell \neq p} \Phi_{\ell}$, where $\Phi_{\ell} = \mathbf{1}_{\Z_{\ell}^2}$ if $\ell \nmid N = \mathfrak{N}\cap \Z$, and if $v_{\ell}(N) = r > 0$, we define
\begin{equation}\label{E05}
\Phi_{\ell} = (\ell^2-1)\ell^{2r-2}\cdot \mathbf{1}_{\ell^r\Z_{\ell}, 1+\ell^r\Z_{\ell}}.
\end{equation}
By Theorem \ref{T406}, we can then view $\Sh{E}^{\Phi^{(p)}}(\frac{k_1-k_2}{2}+\kappa_{\Sh{U}}-\sigma,\kappa_{\Sh{U}}+\sigma-1;{\chi^{(p)}})$ as a nearly overconvergent family over $\Sh{U}\times \Sh{V}$, where we now use the symbols $\kappa_{\Sh{U}}$ and $\sigma$ for the universal weights over $\Sh{U}$ and $\Sh{V}$ respectively.

\begin{defn}
    Let $e_{\underline{\Pi}^{\vee}} \colon H^1(2\kappa_{\Sh{U}}+k_1, 2-k_2;w)^{\leq h} \to H^1(2\kappa_{\Sh{U}}+k_1, 2-k_2;w)^{\leq h}[\underline{\Pi}^{\vee}]$ be the projection. Let $e^{\dagger, \leq h} = e^{\dagger}e^{\leq h}$. Define the meromorphic function $\Sh{L}^{\text{imp}}_{p,\text{As}}(\underline{\Pi})(\kappa_{\Sh{U}},\sigma)$ on $\Sh{U}\times \Sh{V}$ to be
    \[
    \Sh{L}^{\text{imp}}_{p,\text{As}}(\underline{\Pi})(\kappa_{\Sh{U}},\sigma) = (\star)\cdot \nu_{\underline{\Pi}}\left(e_{\underline{\Pi}^{\vee}}e^{\dagger,\leq h}\iota_*\Sh{E}^{\Phi^{(p)}}\left(\tfrac{k_1-k_2}{2}+\kappa_{\Sh{U}}-\sigma,\kappa_{\Sh{U}}+\sigma+\tfrac{k_1-k_2}{2}-1;\chi^{(p)}\right)\right)
    \]
    for $(\star) = \frac{p+1}{p}\cdot (\sqrt{D})^{1-(k_1+k_2+2\kappa_{\Sh{U}})/2-\sigma}(-1)^{\sigma}$, and $\chi^{(p)}$ the restriction to $\Q$ of the common central character of ${\Pi}[a]$.
\end{defn}

\begin{rem} \ 
    \begin{enumerate}
        \item The function is meromorphic in $\kappa_{\Sh{U}}$ because we used the overconvergent projector $e^{\dagger}$, and for a fixed $\kappa_{\Sh{U}} = a$, it is holomorphic in $\sigma$.
        \item The $(\star)$ factor is included to cancel out various unwanted terms arising from local zeta integrals computing imprimitive Asai $L$-values (see the discussion after Definition 7.3.2 of \cite{grossi2024padicasailfunctionsquadratic}).
        \item We have defined this function on $\Sh{U} \times \Sh{V}$ where $\Sh{V} \subset \Sh{W}$ is an arbitrary affinoid. However, if $\Sh{V} \subseteq \Sh{V}'$ are two affinoids, then it follows from the interpolating property (see below) that the resulting functions are clearly compatible under restriction from $\Sh{U} \times \Sh{V}'$ to $\Sh{U} \times \Sh{V}$, so they patch together to give a function on $\Sh{U} \times \Sh{W}$.
    \end{enumerate}
\end{rem}

The function $\Sh{L}^{\text{imp}}_{p,\text{As}}(\underline{\Pi})(\kappa_1,\kappa_2)$ is designed to interpolate imprimitive Asai $L$-values, whose definition and relation to the primitive Asai $L$-function we now recall.

\subsection{Imprimitive Asai \texorpdfstring{$L$}{L}-function and interpolation formula}
For $\Pi$ as above with central character $\chi$ and level $K_1(\mathfrak{N})$, fix the following data. 
\begin{itemize}
    \item For each prime $\ell$, let $W_{\ell}^{\mathrm{new}}$ be the normalised new vector in the Whittaker model $\Sh{W}(\Pi_{\ell}, \psi_{F,\ell})$;
    \item $W_p = W_{\mathfrak{p}_1,\alpha_1}W_{\mathfrak{p}_2}^{\mathrm{new}}$, where $W_{\mathfrak{p}_1, \alpha_1}$ is the $\mathfrak{p}_1$-stabilisation of $W_{\mathfrak{p}_1}^{\mathrm{new}}$ with $U_{\mathfrak{p}_1}^t$-eigenvalue $\alpha_1$;
    \item $\Phi^{(p)} = \prod_{\ell \neq p}\Phi_{\ell}$ is as in (\ref{E05});
    \item $\Phi_p = \Phi_{p,1,1}$ in the notation of \S\ref{S05}.
\end{itemize}

Let $\ell |N = \mathfrak{N}\cap \Z$ be a prime. Then letting 
\[
Z(W_{\ell}^{\mathrm{new}}, \Phi_{\ell};s) := \int_{Z_{H}N_H\backslash H(\Q_{\ell})} W_{\ell}^{\mathrm{new}}(h)f^{\Phi_{\ell}}(h;\hat{\chi}_{\ell},s)\mathrm{d}h
\]
be the local zeta integral at $\ell$, \cite[Proposition 6.4.2]{grossi2024padicasailfunctionsquadratic} shows that there exists a polynomial $C_{\ell}(W_{\ell}^{\mathrm{new}}, \Phi_{\ell};X) \in \C[X,X^{-1}]$ generating the unit ideal, such that
\[
Z(W_{\ell}^{\mathrm{new}}, \Phi_{\ell};s) = L_{\mathrm{As}}(\Pi_{\ell}, s)\cdot C_{\ell}(W_{\ell}^{\mathrm{new}}, \Phi_{\ell};\ell^{-s}).
\]

\begin{defn}
    We define the imprimitive Asai $L$-function to be $L^{\mathrm{imp}}_{\mathrm{As}}(\Pi,s) = \prod_{\ell} L^{\mathrm{imp}}_{\mathrm{As}}(\Pi_{\ell},s)$ where the local $L$-factors are defined as 
    \begin{align*}
     L^{\mathrm{imp}}_{\mathrm{As}}(\Pi_{\ell},s) = 
        \begin{cases}
        L_{\mathrm{As}}(\Pi_{\ell}, s)\cdot C_{\ell}(W_{\ell}^{\mathrm{new}}, \Phi_{\ell};\ell^{-s}) \qquad \qquad &\text{if }\ell|N \\
        L_{\mathrm{As}}(\Pi_{\ell}, s) &\text{otherwise}.
        \end{cases}
    \end{align*}
\end{defn}

Let $\alpha_i^\circ, \beta_i^{\circ}$ be the Satake parameters of $\Pi_{\mathfrak{p}_i}$ (unitarily normalised, so that $|\alpha_i^{\circ}| = |\beta_i^{\circ}| = 1)$. We label the parameters such that $\alpha_1 = p^{(k_1-1)/2}\alpha_1^{\circ}$ is the $U_{\mathfrak{p}_1}^t$-eigenvalue of $W_{\mathfrak{p}_1,\alpha_1}$, so that letting $\beta_1 = p^{(k_1-1)/2}\beta_1^{\circ}$, we have 
\[
W_{\mathfrak{p}_1,\alpha_1} = \left(1 - \frac{\beta_1}{U_{\mathfrak{p}_1}^t}\right)W_{\mathfrak{p}_1}^{\mathrm{new}}.
\]

\begin{defn}
    For $s \in \Z$, define 
    \begin{align*}
        \Sh{E}_p(\mathrm{As}(\Pi), s) = 
            \left(1- \frac{p^{s-1}}{\alpha_1^{\circ}\alpha_2^{\circ}}\right)\left( 1 - \frac{p^{s-1}}{\alpha_1^{\circ}\beta_2^{\circ}}\right)\left(1 - \frac{\beta_1^{\circ}\alpha_2^{\circ}}{p^s}\right)\left(1 - \frac{\beta_1^{\circ}\beta_2^{\circ}}{p^s}\right) 
    \end{align*}
\end{defn}

\begin{prop}[{cf. \cite[Proposition 8.4.1]{grossi2024padicasailfunctionsquadratic}}]
    With the above choices, we have
    \[
    Z(W_p,\Phi_p;s) = \frac{p}{p+1}\Sh{E}_p(\mathrm{As}(\Pi),s)L_{\mathrm{As}}(\Pi_p, s).
    \]
\end{prop}

\begin{theorem}\label{thm:main}
    Let $\underline{\Pi}^{\vee}$ be a Coleman family as above. Then for any $a \geq 0$, and $s \in \Z$ satisfying $1-\frac{(k_1+2a-k_2)}{2} \leq s \leq \frac{k_1+2a-k_2}{2}$, denoting the unitarily normalised $U_{\mathfrak{p}_1}^t$-eigenvalue of $\Pi[a]$ by $\alpha_1^{\circ}$, we have
    \begin{equation*}
        \Sh{L}^{\mathrm{imp}}_{p,\mathrm{As}}(\underline{\Pi})(a,s) = \Sh{E}_p(\mathrm{As}(\Pi[a]), s) \cdot \frac{\Gamma(s+a-1+\tfrac{k_1+k_2}{2})\Gamma(s+a+\tfrac{k_1-k_2}{2})}{2^{(k_1+2a-2)}i^{(1-k_2)}(-2\pi i)^{(2s+2a+k_1-1)}}\cdot L^{\mathrm{imp}}_{\mathrm{As}}(\Pi[a],s)\cdot \frac{\Omega_p(\Pi[a])}{\Omega_{\infty}(\Pi[a])}.
    \end{equation*}
\end{theorem}

\begin{proof}
    The pairing $(\nu_{\underline{\Pi}})_{|a}\left(e_{\underline{\Pi}^{\vee}}e^{\dagger,\leq h}\iota_*\Sh{E}^{\Phi^{(p)}}\left(\tfrac{k_1-k_2}{2}+a-s,a+s+\tfrac{k_1-k_2}{2}-1;\chi^{(p)}\right)\right)$ is induced from the cup product in $H^2_{\mathrm{dR}}(X_{G,0}(\mathfrak{p}_1)_K, \Sym^{k_1+2a-2}\Sh{H}_{\Sh{A},1}\otimes \Sym^{k_2-2}\Sh{H}_{\Sh{A},2})$ and after multiplying by $\Omega_{\infty}/\Omega_p$ can be computed by a global zeta integral
    \[
    Z(\phi, \Phi,s) = \int_{{Z_H}(\A{}{})H(\Q)\backslash H(\A{}{})}\phi(h)\Sh{E}^{\Phi^{(p)}}\left(\tfrac{k_1-k_2}{2}+a-s,a+s+\tfrac{k_1-k_2}{2}-1;\chi^{(p)}\right)\mathrm{d}h
    \]
    where $\phi \in \Pi[a]$ is the unique function corresponding to $W^{\mathrm{ah},1}_{\infty}W_p\prod_{\ell \neq p}W^{\mathrm{new}}_{\ell}$, i.e. the $\mathfrak{p}_1$-stabilisation of $\phi^{\mathrm{ah},1}_{\mathrm{new}}$ with $U_{\mathfrak{p}_1}^t$-eigenvalue $\alpha_1$, and $\Phi = \Phi_{\infty}^{(k_1+2a-k_2)}\Phi^{(p)}\Phi_{p,1,1}$. The formula then follows from the archimedean zeta integral computation of \cite{grossi2024padicasailfunctionsquadratic}.
\end{proof}

\section{Twisted triple products}

 \subsection{Construction}

  Let $\underline{\Pi}$ be a family of eigenforms for $\GL_{2,F}$ over an affinoid $\mathcal{U}$ as in the previous section; and define $\nu_{\underline{\Pi}}$ and periods $\Omega_\infty, \Omega_p$ as before. The linear functional $\nu_{\underline{\Pi}}$ is defined on the Shimura variety of prime-to-$p$ level $\mathfrak{N}$ (the level at which $\underline{\Pi}$ is new); but we may extend it to a linear functional on cohomology any level $\mathfrak{N}'$ divisible by $\mathfrak{N}$ by composition with the trace map.

  We choose also a cuspidal automorphic representation $\Sigma$ of $\GL_{2,\Q}$, generated by a newform of weight $\ell \ge 2$ and level coprime to $p$, and a root $\alpha'$ of its Hecke polynomial at $p$ with slope $< \ell - 1$. If we choose $\mathfrak{N}'$ so the level of $\Sigma$ divides the ideal $\mathfrak{N}' \cap \Z$, then the newform generating $\sigma$ is a section of a line bundle on $X_{H, 0}(p)$.
  
  Then standard results of Coleman imply that we can find an affinoid $\Sh{U}' \subset \Sh{W}$, and a $p$-adic family of $p$-refined Hecke eigensystems $\underline{\Sigma}$ of weight $\ell + 2\kappa_{\Sh{U}'}$, specializing at $\kappa_{\Sh{U}'}$ to $\Sigma$. Moreover, there is a corresponding family of overconvergent sections
  \[ \omega_{\underline{\Sigma}} \in H^0(\Sh{X}_{H, 0}(p)_{[r, 1]}, \underline{\omega}^{\ell + 2\kappa_{\Sh{U}'}}) \]
  for some $r < 1$, specializing at classical points $b \in \Sh{U}' \cap \Z$ with $\ell + 2b \ge 1$ to a $p$-stablisation of the newform generating $\Sigma[b]$. 

  We now suppose that the central character of $\underline{\Sigma}$ is $(\chi^{(p)})^{-1}$, where $\chi^{(p)}$ is the restriction to $\Q$ of the central character of $\underline{\Pi}$ as before. This implies that $\ell = k_1 + k_2 \bmod 2$.

  \begin{prop}
   Let $\omega^{[p]}_{\underline{\Sigma}}$ be the $p$-depletion of $\omega_{\underline{\Sigma}}$, and let $\tau : \Z_p^\times \to \Sh{O}(\Sh{U} \times \Sh{U}')^\times$ be the character $\kappa_{\Sh{U}} - \kappa_{\Sh{U}'} + (k_1 - k_2 - \ell)/2$. Then $\nabla^{\tau}(\omega^{[p]}_{\underline{\Sigma}})$ is a nearly-overconvergent $p$-adic modular form of weight $k_1 - k_2 + \kappa_{\Sh{U}}$.
  \end{prop}
  
  \begin{proof}
   This follows in the same way as the Eisenstein case in Proposition \ref{T406}.
  \end{proof}

  We may therefore define the product
  \[ \mathcal{L}_p(\underline{\Pi} \times \underline{\Sigma}) = \nu_{\underline{\Pi}}\left( e_{\underline{\Pi}^\vee} e^{\dag, \le h} \iota_* \left(\nabla^{\tau}(\omega^{[p]}_{\underline{\Sigma}})\right) \right), \]
  which is a meromorphic function on $\Sh{U} \times \Sh{U}'$ (holomorphic in the second variable, as before).

 \subsection{Interpolation property}

  By construction, at points $(a, b) \in \Z^2$ such that $k_1 - k_2 + 2a \ge \ell + 2b \ge 2$, the specialization of $\mathcal{L}_p(\underline{\Pi} \times \underline{\Sigma})$ is given by
  \[ \frac{\Omega_p(\Pi[a])}{\Omega_\infty(\Pi[a])} \langle \iota^* \eta_{\Pi[a], \mathfrak{p}_1}, \nabla^t \omega_{\Sigma[b]}^{[p]}\rangle_{X_{H, 0}(p)},\]
  where $t = a - b + (k_1 - k_2 - \ell)/2 \ge 0$. A standard computation shows this is equal to
  \[ \frac{\Omega_p(\Pi[a])}{\Omega_\infty(\Pi[a])} \cdot \mathcal{E}_p(\Pi[a] \times \Sigma[b]) \cdot \langle \iota^* \eta_{\Pi[a]}, \nabla^t \omega_{\Sigma[b]}\rangle_{X_{H}} \]
  where $X_H$ is the modular curve of prime-to-$p$ level, and $\mathcal{E}_p(\Pi[a] \times \Sigma[b])$ is the degree 4 Euler factor defined as follows: we write the local factor $\Pi[a]_{\mathfrak{p}_1}$ as an induced representation $\pi(\mu, \nu)$ with $\mu$ the unitary unramified character corresponding to our choice of $p$-stabilisation, and set
  \[ \mathcal{E}_p(\Pi[a] \times \Sigma[b]) = L(\nu \times \Pi[a]_{\mathfrak{p}_2} \times \Sigma[b]_p, \tfrac{1}{2})^{-1}.\] 
  The cup-product $\langle \iota^* \eta_{\Pi[a]}, \nabla^t \omega_{\Sigma[b]}\rangle_{X_{H}}$ can be written as an automorphic period integral
  \[ 
   \int_{Z_H(\A{}{}) H(\Q) \backslash H(\A{}{})} \phi^{\mathrm{ah}, 1}_{\Pi[a], \mathrm{new}} (\iota(h))  \cdot (\delta^t \phi_{\Sigma[b], \mathrm{new}})(h)\ \mathrm{d}h
  \]
  where $\phi_{\Sigma[b], \mathrm{new}}$ is the holomorphic newform generating $\Sigma[b]$, and $\delta^t$ is the Maass--Shimura derivative. If the local root numbers satisfy the condition of \cite[Theorem 3.2]{michele} at all finite places $v$ (see also the subsequent remark in loc. cit. for sufficient conditions in terms of the levels of $\underline{\Pi} $ and $\underline{\Sigma}$ for this to hold), Ichino's formula \cite{Ichino} expresses the square of this period integral as an explicit multiple of $L(\Pi[a] \times \Sigma[b], \tfrac{1}{2})$.
  
  \subsubsection*{Remark}
   The correction terms in Ichino's formula include a product over local integrals, almost all of which are 1. If we form the period integral using new-vectors at all finite places, as above, then it may occur that some of these local terms are 0. However, this can be remedied by replacing the newforms $ \phi^{\mathrm{ah}, 1}_{\Pi[a], \mathrm{new}}$ with its translate by some element $\gamma \in G(\A{\mathrm{f}}{(p)})$, giving a more general $p$-adic $L$-function $\mathcal{L}_p(\underline{\Pi} \times \underline{\Sigma}; \gamma)$. Since the translates of $\phi^{\mathrm{ah}, 1}_{\Pi[a], \mathrm{new}}$ span the automorphic representation, and the local integrals are non-zero for some choice of test data, we deduce that if $L(\Pi[a] \times \Sigma[b], \tfrac{1}{2}) \ne 0$, then there is some $\gamma$ for which $\mathcal{L}_p(\underline{\Pi} \times \underline{\Sigma}; \gamma)(a, b) \ne 0$. Compare \S 5.4 of \cite{LoefflerZerbesGGP}.

\printbibliography

@article {BoxerPilloni_ModularCurve,
    AUTHOR = {Boxer, George and Pilloni, Vincent},
     TITLE = {Higher {H}ida and {C}oleman theories on the modular curve},
   JOURNAL = {\'Epijournal G\'eom. Alg\'ebrique},
  FJOURNAL = {\'Epijournal de G\'eom\'etrie Alg\'ebrique. EPIGA},
    VOLUME = {6},
      YEAR = {2022},
     PAGES = {Art. 16, 33},
      ISSN = {2491-6765},
   MRCLASS = {11F85 (11F33 14G22 14G35)},
  MRNUMBER = {4482376},
MRREVIEWER = {Shaoyun\ Yi},
}

@misc{boxer2021highercolemantheory,
      title={Higher Coleman Theory}, 
      author={George Boxer and Vincent Pilloni},
      year={2021},
      eprint={2110.10251},
      archivePrefix={arXiv},
      primaryClass={math.NT},
      url={https://arxiv.org/abs/2110.10251}, 
}

@misc{grossi2024padicasailfunctionsquadratic,
      title={P-adic Asai L-functions for quadratic Hilbert eigenforms}, 
      author={Giada Grossi and David Loeffler and Sarah Livia Zerbes},
      year={2024},
      eprint={2307.07004},
      archivePrefix={arXiv},
      primaryClass={math.NT},
      url={https://arxiv.org/abs/2307.07004}, 
}

@article {DelignePappas,
    AUTHOR = {Deligne, Pierre and Pappas, Georgios},
     TITLE = {Singularit\'{e}s des espaces de modules de {H}ilbert, en les
              caract\'{e}ristiques divisant le discriminant},
   JOURNAL = {Compositio Math.},
  FJOURNAL = {Compositio Mathematica},
    VOLUME = {90},
      YEAR = {1994},
    NUMBER = {1},
     PAGES = {59--79},
      ISSN = {0010-437X},
   MRCLASS = {11F41 (14G35 14K10)},
  MRNUMBER = {1266495},
MRREVIEWER = {Burt Totaro},
       URL = {http://www.numdam.org/item?id=CM_1994__90_1_59_0},
}

@article {Pappas,
    AUTHOR = {Pappas, Georgios},
     TITLE = {Arithmetic models for {H}ilbert modular varieties},
   JOURNAL = {Compositio Math.},
  FJOURNAL = {Compositio Mathematica},
    VOLUME = {98},
      YEAR = {1995},
    NUMBER = {1},
     PAGES = {43--76},
      ISSN = {0010-437X,1570-5846},
   MRCLASS = {11F41 (14F30 14G35 14K10)},
  MRNUMBER = {1353285},
MRREVIEWER = {Burt\ Totaro},
       URL = {http://www.numdam.org/item?id=CM_1995__98_1_43_0},
}

@article {Fargues10,
    AUTHOR = {Fargues, Laurent},
     TITLE = {La filtration de {H}arder-{N}arasimhan des sch\'emas en
              groupes finis et plats},
   JOURNAL = {J. Reine Angew. Math.},
  FJOURNAL = {Journal f\"ur die Reine und Angewandte Mathematik. [Crelle's
              Journal]},
    VOLUME = {645},
      YEAR = {2010},
     PAGES = {1--39},
      ISSN = {0075-4102,1435-5345},
   MRCLASS = {14L15},
  MRNUMBER = {2673421},
MRREVIEWER = {Alan\ Koch},
       DOI = {10.1515/CRELLE.2010.058}
}

@article {KatzNicholasM1978pLfC,
    AUTHOR = {Katz, Nicholas M.},
     TITLE = {{$p$}-adic {$L$}-functions for {CM} fields},
   JOURNAL = {Invent. Math.},
  FJOURNAL = {Inventiones Mathematicae},
    VOLUME = {49},
      YEAR = {1978},
    NUMBER = {3},
     PAGES = {199--297},
   MRCLASS = {10D25 (12A65 12A67 14K22)},
  MRNUMBER = {513095},
MRREVIEWER = {V. V. Shokurov},
       DOI = {10.1007/BF01390187},
}

@article {Andreatta2018leHS,
    AUTHOR = {Andreatta, Fabrizio and Iovita, Adrian and Pilloni, Vincent},
     TITLE = {Le halo spectral},
   JOURNAL = {Ann. Sci. \'{E}c. Norm. Sup\'{e}r. (4)},
  FJOURNAL = {Annales Scientifiques de l'\'{E}cole Normale Sup\'{e}rieure. Quatri\`eme
              S\'{e}rie},
    VOLUME = {51},
      YEAR = {2018},
    NUMBER = {3},
     PAGES = {603--655},
   MRCLASS = {11F33 (11F41 11F85 11G18 14G35)},
  MRNUMBER = {3831033},
MRREVIEWER = {Sarah Livia Zerbes},
       DOI = {10.24033/asens.2362},
}

@inproceedings {katzpadic,
    AUTHOR = {Katz, Nicholas M.},
     TITLE = {{$p$}-adic properties of modular schemes and modular forms},
 BOOKTITLE = {Modular functions of one variable, {III} ({P}roc. {I}nternat.
              {S}ummer {S}chool, {U}niv. {A}ntwerp, {A}ntwerp, 1972)},
    SERIES = {Lecture Notes in Math., Vol. 350},
     PAGES = {69--190},
 PUBLISHER = {Springer, Berlin},
      YEAR = {1973},
   MRCLASS = {10D15 (14D20)},
  MRNUMBER = {0447119},
MRREVIEWER = {V. V. Sokurov},
}

@article {Fargues11,
    AUTHOR = {Fargues, Laurent},
     TITLE = {La filtration canonique des points de torsion des groupes
              {$p$}-divisibles},
      NOTE = {With collaboration of Yichao Tian},
   JOURNAL = {Ann. Sci. \'Ec. Norm. Sup\'er. (4)},
  FJOURNAL = {Annales Scientifiques de l'\'Ecole Normale Sup\'erieure.
              Quatri\`eme S\'erie},
    VOLUME = {44},
      YEAR = {2011},
    NUMBER = {6},
     PAGES = {905--961},
      ISSN = {0012-9593,1873-2151},
   MRCLASS = {14L05},
  MRNUMBER = {2919687},
MRREVIEWER = {Eva\ Viehmann},
       DOI = {10.24033/asens.2157},
}

@misc{kazi2024twistedtripleproductpadic,
      title={Twisted Triple Product $p$-adic $L$-function for Finite Slope Families of Hilbert Modular Forms}, 
      author={Ananyo Kazi},
      year={2024},
      eprint={2401.13230},
      archivePrefix={arXiv},
      primaryClass={math.NT},
      url={https://arxiv.org/abs/2401.13230}, 
}

@article {andreatta2021triple,
    AUTHOR = {Andreatta, Fabrizio and Iovita, Adrian},
     TITLE = {Triple product {$p$}-adic {$L$}-functions associated to finite
              slope {$p$}-adic families of modular forms},
   JOURNAL = {Duke Math. J.},
  FJOURNAL = {Duke Mathematical Journal},
    VOLUME = {170},
      YEAR = {2021},
    NUMBER = {9},
     PAGES = {1989--2083},
   MRCLASS = {11F66 (11F85)},
  MRNUMBER = {4278669},
MRREVIEWER = {Ravi Raghunathan},
       DOI = {10.1215/00127094-2020-0076},
}

@misc{graham2023padic,
      title={p-adic interpolation of Gauss--Manin connections on nearly overconvergent modular forms and p-adic L-functions}, 
      author={Andrew Graham and Vincent Pilloni and Joaquín Rodrigues Jacinto},
      year={2023},
      eprint={2311.14438},
      archivePrefix={arXiv},
      primaryClass={math.NT}
}

@article {FakhruddinPilloni,
    AUTHOR = {Fakhruddin, Najmuddin and Pilloni, Vincent},
     TITLE = {Hecke operators and the coherent cohomology of {S}himura
              varieties},
   JOURNAL = {J. Inst. Math. Jussieu},
  FJOURNAL = {Journal of the Institute of Mathematics of Jussieu. JIMJ.
              Journal de l'Institut de Math\'ematiques de Jussieu},
    VOLUME = {22},
      YEAR = {2023},
    NUMBER = {1},
     PAGES = {1--69},
      ISSN = {1474-7480,1475-3030},
   MRCLASS = {11F46 (11F60 11G18 14G35)},
  MRNUMBER = {4556929},
MRREVIEWER = {Hideshi\ Takayanagi},
       DOI = {10.1017/S1474748021000050},
}

@book {GorenHilbertbook,
    AUTHOR = {Goren, Eyal Z.},
     TITLE = {Lectures on {H}ilbert modular varieties and modular forms},
    SERIES = {CRM Monograph Series},
    VOLUME = {14},
      NOTE = {With the assistance of Marc-Hubert Nicole},
 PUBLISHER = {American Mathematical Society, Providence, RI},
      YEAR = {2002},
     PAGES = {x+270},
      ISBN = {0-8218-1995-X},
   MRCLASS = {11F41 (11F33 11G10 11G18 14G35)},
  MRNUMBER = {1863355},
MRREVIEWER = {Chandrashekhar\ Khare},
       DOI = {10.1090/crmm/014},
}

@article {LPSZ,
    AUTHOR = {Loeffler, David and Pilloni, Vincent and Skinner, Christopher
              and Zerbes, Sarah Livia},
     TITLE = {Higher {H}ida theory and {$p$}-adic {$L$}-functions for
              {$\mathrm{GSp}_4$}},
   JOURNAL = {Duke Math. J.},
  FJOURNAL = {Duke Mathematical Journal},
    VOLUME = {170},
      YEAR = {2021},
    NUMBER = {18},
     PAGES = {4033--4121},
      ISSN = {0012-7094,1547-7398},
   MRCLASS = {11F67 (11F46 11R23)},
  MRNUMBER = {4348233},
MRREVIEWER = {Chan-Ho\ Kim},
       DOI = {10.1215/00127094-2021-0049},
}

@article {LoefflerRankin,
    AUTHOR = {Loeffler, David},
     TITLE = {A note on {$p$}-adic {R}ankin-{S}elberg {$L$}-functions},
   JOURNAL = {Canad. Math. Bull.},
  FJOURNAL = {Canadian Mathematical Bulletin. Bulletin Canadien de
              Math\'ematiques},
    VOLUME = {61},
      YEAR = {2018},
    NUMBER = {3},
     PAGES = {608--621},
      ISSN = {0008-4395,1496-4287},
   MRCLASS = {11F85 (11F67 11G40 14G35)},
  MRNUMBER = {3831932},
MRREVIEWER = {Giovanni\ Rosso},
       DOI = {10.4153/CMB-2017-047-9},
}

@misc{grahamrockwood24,
      title={Nearly higher Coleman theory and p-adic L-functions for $\mathrm{GSp}(4) \times \mathrm{GL}(2)$ and $\mathrm{GSp}(4) \times \mathrm{GL}(2) \times \mathrm{GL}(2)$}, 
      author={Andrew Graham and Rob Rockwood},
      year={2024},
      eprint={2411.04559},
      archivePrefix={arXiv},
      primaryClass={math.NT}, 
}

@article {Ichino,
    AUTHOR = {Ichino, Atsushi},
     TITLE = {Trilinear forms and the central values of triple product
              {$L$}-functions},
   JOURNAL = {Duke Math. J.},
  FJOURNAL = {Duke Mathematical Journal},
    VOLUME = {145},
      YEAR = {2008},
    NUMBER = {2},
     PAGES = {281--307},
      ISSN = {0012-7094},
   MRCLASS = {11F67 (11F70)},
  MRNUMBER = {2449948},
MRREVIEWER = {Min Ho Lee},
       DOI = {10.1215/00127094-2008-052},
       URL = {https://doi.org/10.1215/00127094-2008-052},
}

@misc{LoefflerZerbesGGP,
      title={P-adic L-functions and diagonal cycles for $GSp(4) \times GL(2) \times GL(2)$}, 
      author={David Loeffler and Sarah Livia Zerbes},
      year={2021},
      eprint={2011.15064},
      archivePrefix={arXiv},
      primaryClass={math.NT}, 
}

@article {michele,
    AUTHOR = {Blanco-Chac\'{o}n, Iv\'{a}n and Fornea, Michele},
     TITLE = {Twisted triple product {$p$}-adic {$L$}-functions and
              {H}irzebruch-{Z}agier cycles},
   JOURNAL = {J. Inst. Math. Jussieu},
  FJOURNAL = {Journal of the Institute of Mathematics of Jussieu. JIMJ.
              Journal de l'Institut de Math\'{e}matiques de Jussieu},
    VOLUME = {19},
      YEAR = {2020},
    NUMBER = {6},
     PAGES = {1947--1992},
      ISSN = {1474-7480},
   MRCLASS = {11F41 (11F67 11G35 11G40)},
  MRNUMBER = {4166999},
MRREVIEWER = {Kimball L. Martin},
       DOI = {10.1017/s1474748019000021},
       URL = {https://doi.org/10.1017/s1474748019000021},
}
\end{document}